\pdfoutput=1 
\documentclass[msnc,nonblind]{informs_modified} 

\OneAndAHalfSpacedXI

\usepackage{natbib}
 \bibpunct[, ]{(}{)}{,}{a}{}{,}%
 \def\bibfont{\small}%
\usepackage{pifont}
\usepackage{xspace}

\newcommand{\da}{\text{\ding{172}}\xspace}
\newcommand{\db}{\text{\ding{173}}\xspace}
\newcommand{\dc}{\text{\ding{174}}\xspace}

\newcommand{\x}{\textcolor{black}}

\newcommand{\bE}{\mathbb{E}}
\newcommand{\bP}{\mathbb{P}}
\newcommand{\bI}{\mathbf{1}}
\newcommand{\cS}{\cal{S}}
\newcommand{\cT}{\cal{T}}
\newcommand{\OFF}{\mathsf{OFF}}
\newcommand{\ON}{\mathsf{ON}}
\newcommand{\hpolicy}{\mathsf{STP}}

\TheoremsNumberedThrough     
\ECRepeatTheorems

\EquationsNumberedThrough    

\begin{document}

\RUNAUTHOR{Truong, Wang}
\RUNTITLE{}
\TITLE{Prophet Inequality with Correlated Arrival Probabilities, with Application to Two Sided Matchings}
\ARTICLEAUTHORS{%
\AUTHOR{Van-Anh Truong, Xinshang Wang} \AFF{Department
of Industrial Engineering and Operations Research, Columbia
University, New York, NY, USA, \EMAIL{vatruong@ieor.columbia.edu, xw2230@columbia.edu}
\URL{}}
}
\ABSTRACT{The classical Prophet Inequality arises from a fundamental problem in optimal-stopping theory.  In this problem, a gambler sees a finite sequence of independent, non-negative random variables.  If he stops the sequence at any time, he collects a reward equal to the most recent observation. The Prophet Inequality states that, knowing the distribution of each random variable, the gambler can achieve at least half as much reward in expectation, as a prophet who knows the entire sample path of random variables \citep{krengel1978semiamarts}. \x{In this paper, we prove a corresponding bound for \emph{correlated} non-negative random variables.} 
We analyze two methods for proving the bound, a constructive approach, which produces a worst-case instance, and a reductive approach, which characterizes a certain submartingale arising from the reward process of our online algorithm. 

We apply this new prophet inequality to the design of algorithms for a class of two-sided bipartite matching problems that underlie \emph{online task assignment problems}.  In these problems, demand units of various types arrive randomly and sequentially over time according to some stochastic process.  Tasks, or supply units, arrive according to another stochastic process. Each demand unit must be irrevocably matched to a supply unit or rejected.  The match earns a reward that depends on the pair.  The objective is to maximize the total expected reward over the planning horizon.  The problem arises in mobile crowd-sensing and crowd sourcing contexts, where workers and tasks must be matched by a platform according to various criteria.  We derive the first online algorithms with worst-case performance guarantees for our class of two-sided bipartite matching problems. 
}

\maketitle
\section{Introduction}
The classical Prophet Inequality arises from a fundamental problem in optimal-stopping theory.  In this problem, a gambler sees a finite sequence of independent, non-negative random variables.  If he stops the sequence at any time, he collects a reward equal to the most recent observation. The Prophet Inequality states that, knowing the distribution of each random variable, the gambler can achieve at least half as much reward in expectation, as a prophet who knows the entire sample path of random variables \citep{krengel1978semiamarts}.  The classical prophet inequality with independent random variables was proved by \cite{krengel1977}.  Its importance arises from its role as a primitive in a wide range of decision problems.  Since the appearance of the first result, various versions of the prophet inequality has been proved. \cite{hill1983stop} study the inequality for independent, uniformly bounded random variables.   \cite{rinott1987comparisons} prove a version for bounded negatively-dependent random variables.  \cite{samuel1991prophet} obtain general results for negatively dependent random variables, and provide some examples for the case of positively dependent variables.  
\x{In this paper, we study a version of prophet inequality where the sequence of random variables are modeled as a customer arrival process and can be arbitrarily correlated. The specification of our prophet inequality will be made clear in the problem formulation.}

We apply this new prophet inequality to the design of algorithms for a class of two-sided bipartite matching problems underlying \emph{online task assignment problems} (OTA).  In these problems, demand units of various types arrive randomly and sequentially over time according to some stochastic process.  Tasks, or supply units, arrive according to another stochastic process. Each demand unit must be irrevocably matched to a supply unit or rejected.  The match earns a reward that depends on the pair.  The objective is to maximize the total expected reward over the planning horizon.  The problem arises in mobile crowd-sensing and crowd-sourcing contexts, where workers and tasks that arrive randomly overtime and must be matched by a platform according to various criteria.  For example, the marketplaces Upwork, Fiverr, and Freelancer  match providers with customers of professional services.  Walmart evaluated a proposal to source its own customers to deliver orders \citep{barr2013walmart}.  The mobile platforms Sensorly, Vericell, VTrack, and PIER  outsource the task of collecting analyzing data, called \emph{sensing}, to millions of mobile users.  Enabled by information technology, these crowd-sourcing and crowd-sensing businesses are revolutionalizing the traditional marketplace.  For example, Freelancers constitute 35\% of the U.S. workforce and have generated a trillion dollars in income as of 2015 \citep{pofeldt2016freelancers}.  A survey found that 73\% of freelancers have found work more easily because of technology \citep{pofeldt2016freelancers}.


The \emph{two-sided matching problem} that underlies many examples of OTA is very difficult to solve optimally, due to three main reasons.  First, given the many characteristics of both demand and supply types and their importance in determining the quality of a match, the decision problem must keep track of a vast amount of information, including the current state of supply and future demand and supply arrivals.  Second, both demand and supply processes may change over time, so that the decision-making environment might be constantly changing.  Third, demand units tend to be time-sensitive and unmatched supply units might also leave the system after a time, so that they have a finite period of availability.

Our contributions in this paper are as follow:
\begin{itemize}
\item \x{We prove the first prophet inequality for a class of arbitrarily correlated non-negative random variables.}
We analyze two methods for proving the bound, a constructive approach, which produces a worst-case instance, and a reductive approach, which characterizes a certain submartingale arising from the reward process of our online algorithm.  

\item We formulate a new model of bipartite matching with non-homogenous Poisson arrivals for both demand and supply units.  Supply units can wait a deterministic amount of time, whereas demand units must be matched irrevocably upon arrival.  Decisions are not batched and must be made for one demand unit at at time.  Our model underlies an important class of online task assignment problems for crowd-sourcing and crowd-sensing applications.

\item We derive the first online algorithms with worst-case performance guarantees for our class of two-sided bipartite matching problems.  We prove that our algorithms have expected reward no less than $\frac{1}{4}$ times that of an optimal offline policy, which knows all demand and supply arrivals upfront and makes optimal decisions given this information. 

\item We provide numerical experiments showing that despite the conservative provable ratio of $1/4$, our online algorithm captures about half of the offline expected reward.  We propose improved algorithms that \x{in the experiments} capture $65\%$ to $70\%$ of the offline expected reward. Moreover, the improved algorithms outperform the greedy and the bid-price heuristics in all scenarios. These results demonstrate the advantage of using our online algorithms as they have not only optimized performance in the worst-case scenario, but also satisfactory performance in average-case scenarios.
\end{itemize}

\section{Literature Review}  
We review five streams of literature that are most closely related to our problem class.

\subsection{Static matching}
A variety of matching problems have been studied in static settings, for example, college-admissions problems, marriage problems, and static assignment problems.  In these problems, the demand and supply units are known.  The reward of matching each demand with each supply unit is also known.  The objective is to find a maximum-reward matching.  See \cite{abdulkadiroglu2013matching} for a recent review.  Our setting differs in that demand units arrive randomly over time and decisions must be made before all the units have been fully observed.

\subsection{Dynamic assignment}
Dynamic-assignment problems are a class of problems in which a set of resources must be dynamically assigned to a stream of tasks that randomly arrive over time.  These problems have a long history, beginning with Derman, Lieberman and Ross (1972)\nocite{derman1972sequential}.  See \cite{su2005patient} for a recent review of this literature.

\cite{spivey2004dynamic} study a version of the dynamic assignment problem in which the resources may arrive randomly over time.  They develop approximate-dynamic-programming heuristics for the problem.  They do not derive performance bounds for their heuristics.

\cite{anderson2013efficient} study a specialized model in which supply and demand units are identical, and arrivals are stationary over time.  They characterize the performance of the greedy policy under various structures for the demand-supply graph, where the objective is to minimize the total waiting time for all supply units.

\cite{akbarpour2014dynamic} analyze a dynamic matching problem for which they derive several broad insights.   They analyze two algorithms for which they derive bounds on the relative performance under various market conditions.  Their work differs from ours in four ways.  First, they assume that arrivals are stationary where as we allow non-stationary arrivals.  Second, they assume that demand units are identical except for the time of arrivals and supply units are also identical except for the time of arrival, whereas we allow heterogeneity among the units.  Finally, they study an unweighted matching problem in which each match earns a unit reward, whereas we study a more general weighted matching problem.  Finally, their results hold in asymptotic regimes, where the market is large and the horizon is long, whereas our results hold in any condition.

More recently  \cite{hu2015dynamic} study a dynamic assignment problem for two-sided markets similar to ours. They also allow for random, non-stationary arrivals of demand and supply units.  They derive structural results for the optimal policy and asymptotic bounds.  We depart from both of the above papers in focusing on providing algorithms with theoretical performance guarantees on all problem instances. 

\cite{baccara2015optimal} study a dynamic matching problem in which demand units can wait, and there is a tradeoff between waiting for a higher-quality match, and incurring higher waiting costs.  Their setting is limited to just two types of units (demand or supply), whereas we allow arbitrarily many types.  They also assume stationary arrivals whereas we allow non-stationary arrivals.

\subsection{Online Matching}
In online matching, our work fundamentally extends the class of problems that have been widely studied. 
In existing online matching problems, the set of available supply units is known and corresponds to one set of nodes. Demand units arrive one by one, and correspond to a second set of nodes. As each demand node arises, its adjacency to the resource nodes is revealed. Each edge has an associated weight. The system must match each demand node irrevocably to an adjacent supply node. The goal is to maximize the total weighted or unweighted size of the matching.  

When demands are chosen by an adversary, the online \emph{unweighted} bipartite matching problem is originally shown by Karp, Vazirani and Vazirani (1990)\nocite{karp1990optimal} to have a worst-case relative reward of $0.5$ for deterministic algorithms and $1-1/e$ for randomized algorithms.  The \emph{weighted} this problem cannot be bounded by any constant \citep{mehta2012online}. Many subsequent works have tried to design algorithms with bounded relative reward for this problem under more regulated demand processes.

Three types of demand processes have been studied. The first type of demand processes studied is one in which each demand node is independently and identically chosen with replacement from a \emph{known} set of nodes. Under this assumption, \citep{jaillet2013online, manshadi2012online, bahmani2010improved, feldman2009online} propose online algorithms with worst-case relative reward higher than $1-1/e$ for the unweighted problem. Haeupler, Mirrokni, Vahab and Zadimoghaddam (2011)\nocite{haeupler2011online} study online algorithms with worst-case relative reward higher than $1-1/e$ for the weighted bipartite matching problem.

The second type of demand processes studied is one in which the demand nodes are drawn randomly without replacement from an unknown set of nodes.  This assumption has been used in the secretary problem (Kleinberg 2005, Babaioff, Immorlica, Kempe, and Kleinberg 2008)\nocite{kleinberg2005multiple,babaioff2008online}, ad-words problem \citep{goel2008online} and bipartite matching problem (Mahdian and Yan 2011, Karande, Mehta, and Tripathi 2011)\nocite{mahdian2011online,karande2011online}.

A variation to the second type of demand processes studied is one in which each demand node requests a very small amount of resource. This assumption, called the \emph{small bid} assumption, together with the assumption of randomly drawn demands, lead to polynomial-time approximation schemes (PTAS) for problems such as ad-words \citep{Devanur09theadwords}, stochastic packing (Feldman, Henzinger, Korula, Mirrokni, and Stein 2010)\nocite{feldman2010online}, online linear programming (Agrawal, Wang and Ye 2009)\nocite{agrawal2009dynamic}, and packing problems \citep{molinaro2013geometry}. Typically, the PTAS proposed in these works use dual prices to make allocation decisions. Devanur, Jain, Sivan, and Wilkens (2011)\nocite{devanur2011near} study a resource-allocation problem in which the distribution of nodes is allowed to change over time, but still needs to follow a requirement that the distribution at any moment induce a small enough offline objective value. They then study the asymptotic performance of their algorithm. In our model, the amount capacity requested by each customer is not necessary small relative to the total amount of capacity available.

The third type of demand processes studied are \emph{independent}, non-homogenous Poisson processes. \x{Alaei, Hajiaghayi and Liaghat (2012)\nocite{alaei2012online}, Wang, Truong and Bank (2015)\nocite{wangTB2015} and Stein, Truong and Wang (2017)\nocite{SteinTW2015} propose online algorithms for online allocation problems. } 
\x{ We depart from these papers in two major ways. The algorithms in these papers consist of two main steps.  In the first step, they solve a deterministic assignment LP to find the probabilities of routing each demand to each supply unit.  Given this routing, in the second step, they make an online decision} to determine whether to match a routed demand unit to a supply unit at any given time.    In contrast, in the first step, we find \emph{conditional probabilities} of routing demand to supply units, given the set of supply units that have arrived at any given time.  We approximate these conditional probabilities because they are intractable to compute directly.  In the second step, after routing demands to supply units according to these conditional probabilities, \x{we design an admission algorithm based on the solution of our prophet inequality that, unlike existing admission techniques, deals with  \emph{correlated demand arrivals}.}

\subsection{Online Task Assignment}
This is a subclass of online matching problems that has seen an explosion of interest in recent years.  Almost all of these works model either the tasks or the workers as being fixed.  \cite{ho2012online, assadi2015online, hassan2014multi, manshadi2012online} study variations of OTA problems.  \cite{singer2013pricing} consider both pricing and allocation decisions for OTA.  \cite{singla2013truthful} study both learning and allocation decisions for OTA.  \cite{zhao2014crowdsource, subramanian2015online} study auction mechanism for OTA.

\cite{tong2016online} study OTA when the arrivals of both workers and tasks are in random order.  Their algorithms achieve a competitive ratio of $1/4$.  Concurrent with our work, \cite{dickerson2018assigning} study a similar model with two-sided, i.i.d. arrivals.  They prove that a non-adaptive algorithm achieves a competitive ratio of 0.295.  Further, they show that no online algorithm can achieve a ratio better than 0.581, even if all rewards are the same.  Note that both the models of \cite{tong2016online} and \cite{dickerson2018assigning} are more restrictive than ours.   In a model with time-varying arrivals such as ours, non-adaptive algorithms such as the one proposed by \cite{dickerson2018assigning}, or a greedy algorithms such as one of the algorithms proposed by \cite{tong2016online}, are unlikely to perform well.

\subsection{Revenue Management}

Our work is also related to the revenue management literature. We refer to \cite{TalluriV2004} for a comprehensive review of this literature. 

Our work is related to the still limited literature on designing policies for revenue management that are have worst-case performance guarantees. \cite{ball2009toward} analyze online algorithms for the single-leg revenue-management problem. Their performance metric compares online algorithms with an optimal offline algorithm under the worst-case instance of demand arrivals. They prove that the competitive ratio cannot be bounded by any constant when there are arbitrarily many customer types.  Qin, Zhang, Hua and Shi (2015)\nocite{CongApproximationRM} study approximation algorithms for an admission control problem for a single resource when customer arrival processes can be correlated over time.  They prove a constant approximation ratio for the case of two customer types, and also for the case of multiple customer types with specific restrictions. They allow only one type of resource to be allocated. Gallego, Li, Truong and Wang (2015)\nocite{gallegoLTW2015} study online algorithms for a personalized choice-based revenue-management problem.  They allow multiple customer types and products, and non-stationary independent demand arrivals.  They allow customers to select from assortments of offered products according to a general choice model.  They prove that an LP-based policy earns at least half of the expected revenue of an optimal policy that has full hindsight.

\section{Prophet Inequality with Correlated Arrivals}
\subsection{Problem Formulation} \label{sec:prophetModel}

Throughout this paper, we let $[k]$ denote the set $\{1,2,\ldots,k\}$ for any positive integer $k$.

Consider a finite planning horizon of $T$ periods.
There are $I$ customer types and one unit of a single resource that is managed by some platform. 
In each period $t$, depending on exogenous state information $S_t$ that is observable by time $t$, a customer of type $i$ will arrive with some probability $p_{it}(S_t)$. Upon an arrival of a customer, the platform can either sell the resource to the customer, or irrevocably reject the customer.  The reward earned by the platform for selling the resource to a customer of type $i \in [I]$ in period $t \in [T]$ is $r_{it}\geq 0$.  The goal of the platform is to maximize the expected total reward collected over the planning horizon.

Unlike existing research assuming independent or stationary arrival distributions, we allow the sequence of arrival probabilities $(p_{1t}(S_t), p_{2t}(S_t),\ldots, p_{It}(S_t))_{t=1,\ldots,T}$ to be a correlated stochastic process, which depends on the sample path $(S_1, S_2,\ldots,S_T)$ that is realized.  We assume that we know the joint distribution of this stochastic process of arrival probabilities, that is, the distribution of $\{S_t\}$, and the distribution of arrivals conditional upon $\{S_t\}$.  As a simple example, $S_t$ may represent the weather history at times $\{1,\ldots, t\}$. Our model would capture any correlation in the weather forecast, and assumes that the customer arrival probabilities $p_{it}(S_t)$ in each period $t$ are determined by the weather history $S_t$ up to time $t$.

Let $X_{it} \in \{0,1\}$ be the random variable indicating whether a customer of type $i \in [I]$ will arrive in period $t \in [T]$, in state $S_t$. We must have $\bE[X_{it} | S_t] = p_{it}(S_t)$. We assume the time increments are sufficiently fine so that, similar to many standard models in revenue management \citep{van2005introduction}, we can assume that at most one customer arrives in any given period. More precisely, we apply the common practice in revenue management assuming that $\sum_{i \in [I]} p_{it}(S_t) \leq 1$ almost surely for all $t \in [T]$, and 
\begin{equation}\label{eq:XitDefinition}
X_{it} = \bI\{u_t \in [\sum_{k=1}^{i-1} p_{kt}(S_t), \sum_{k=1}^i p_{kt}(S_t))\},
\end{equation}
where $u_t$ is an independent $[0,1]$ uniform random variable associated with period $t$. As a result, we have $\sum_{i \in [I]} X_{it} \leq 1$ almost surely for all $t \in [T]$. 



%

In each period $t \in [T]$, events take place in the following order:
\begin{enumerate}
\item The platform observes $S_t$, and thus knows $p_{it}(S_t)$ for all $i \in [I]$.
\item The arrivals $(X_{1t}, \ldots, X_{It})$ are realized according to \eqref{eq:XitDefinition}.
\item The platform decides whether to sell the resource to the arriving customer if any.
\end{enumerate}

Let $\cS_t$ denote the support of $S_t$, for all $t \in [T]$. 
\x{Without loss of generality and for ease of notation, we define $S_t$ as the set of external information in all the periods $1,2,\ldots,t$.
As a result, given any $S_t$, the path $(S_1,\ldots, S_t)$ is uniquely determined.
We call a realization of $S_T$ a \emph{sample path}. We assume that we know the joint distribution of $(S_1,\ldots, S_T)$ in the sense that we are able to simulate the sample paths, and able to estimate the expected values of functions of the sample paths.
}

\x{
We can use a tree structure to represent the process $(S_t)_{t \in [T]}$. 
Let $S_0$ be a dummy root node of the tree.
Every realization of $S_1$ is a direct descendant of $S_0$. Recursively, for any tree node that is a realization of $S_t$, its direct descendants are all the different realizations of $S_{t+1}$ conditional on the value of $S_t$.
}

\subsection{Definition of Competitive Ratios}

\x{We will state the prophet inequality with correlated arrivals by proving the \emph{competitive ratio} of an algorithm used by the platform. Specifically,}
define an \emph{optimal offline algorithm} $\OFF$ as an algorithm that knows $(X_{it})_{i \in [I]; t \in [T]}$ at the beginning of period $1$ and makes optimal decisions to sell the resource given this information. \x{By contrast, the platform can use an \emph{online algorithm} to make decisions in each period $t \in [T]$ based on only $S_t$ and $(X_{it'})_{i \in [I]; t' \in [t]}$.}  We use $V^\OFF$ to denote the reward of $\OFF$, and $V^\ON$ the reward of an online algorithm $\ON$. 

\begin{definition}
An online algorithm $\ON$ is \emph{$c$-competitive} if
\[ \bE[V^\ON] \geq c \,\bE[V^\OFF],\]
where the expectation is taken over both $S_T$ and the random arrivals $(X_{it})_{i \in [I]; t \in [T]}$.
\end{definition}

\subsection{Offline Algorithm and Its Upper Bound}

\x{We first show that the expected reward of the offline algorithm can be bounded from above by the expected total reward collected from the entire sample path.}

\begin{proposition} \label{prop:prophetUpperBound}
$\bE[ V^\OFF] \leq \bE[ \sum_{t=1}^T \sum_{i=1}^I r_{it}p_{it}(S_{t})]$.
\end{proposition}
\begin{proof}{Proof.}
Suppose the resource has infinitely many units, so that the platform sells one unit of the resource to every arriving customer. The resulting expected total reward 
\[ \bE[\sum_{i \in [I]} \sum_{t \in [T]} r_{it}X_{it}]  = \sum_{i \in [I]} \sum_{t \in [T]} r_{it}\bE[X_{it}] = \sum_{i \in [I]} \sum_{t \in [T]} r_{it} \bE[ p_{it}(S_t)]  \]
is clearly an upper bound on $\bE[ V^\OFF]$.
\halmos
\end{proof}

\section{Online Algorithm} \label{sec:prophetAlg}

\x{In this section, we propose a simulation-based threshold policy ($\hpolicy$) for the model of prophet inequality with correlated arrivals and prove its performance guarantee. } 

Conditioned on any sample path $S_T$, define 
\[ \cT(S_T):= \sum_{i \in [I]} \sum_{t \in [T]} p_{it}(S_t)\]
as the total expected number of customer arrivals. 

$\hpolicy$ needs to know a uniform upper bound $\bar \cT$ on $\cT(\cdot)$ (i.e., $\bar \cT \geq \cT(S_T)$ with probability one). Given such a $\bar \cT$, $\hpolicy$ computes the threshold 
\begin{equation}\label{eq:hnew}
h(S_t) := \bE\!\!\left[\frac{\sum_{t'=t+1}^T \sum_{i=1}^I r_{it'}p_{it'}(S_{t'})}{1 + \bar \cT -\sum_{t'=1}^t \sum_{i=1}^I p_{it'}(S_{t'})}\ \big|\ S_t\right]
\end{equation}
for deciding whether to sell the resource in period $t$.  Specifically, upon an arrival of a type-$i$ customer in period $t$, if the resource is still available,  $\hpolicy$ sells the resource to the customer if $r_{it} \geq h(S_t)$ and rejects the customer otherwise. Note that $h(\cdot)$ can be computed for each scenario $S_t$ by simulating the sample paths that potentially arise conditional upon $S_t$. In particular, by Proposition \ref{prop:prophetUpperBound} and the definition of $h(\cdot)$, we have \x{(recall that $S_0$ is a (deterministic) dummy variable)}
\begin{equation} \label{eq:hS0}
h(S_0) = \bE\!\!\left[\frac{\sum_{t'=1}^T \sum_{i=1}^I r_{it'}p_{it'}(S_{t'})}{1 + \bar \cT}  \big|\ S_0\right] = \frac{\bE[\sum_{t'=1}^T \sum_{i=1}^I r_{it'}p_{it'}(S_{t'})]}{1 + \bar \cT} \geq \frac{\bE[V^\OFF]}{1 + \bar \cT}.
\end{equation}

Given that $\bE[V^\OFF]$ can be upper-bounded by $h(S_0)$ with a multiplicative factor, the goal of our analysis is to establish a relationship between $h(S_0)$ and the expected reward $\bE[V^\hpolicy]$ of our online algorithm.

%


\subsection{Performance Guarantee}

In this section, we provide two methods for proving the prophet inequality\x{, i.e., the competitive ratio of $\hpolicy$, under correlated arrival probabilities.}  The first method is constructive in the sense that it reasons about the structure of a worst-case problem instance, and exhibits this structure explicitly.  The second method is deductive in that it proves the existence of the bound without shedding light on the worst-case problem instance.

\subsubsection{Constructive method for proving performance bound.}



%

\x{With the constructive method, we focus on proving the competitive ratio of $\hpolicy$ for the case $\bar \cT = 1$. Since all the decisions made by an online algorithm are based on the realized value of $S_1$, we can assume without loss of generality that $S_1$ is deterministic (i.e., the competitive ratio holds when $S_1$ is set to be any of its possible realizations). }

\x{Then we fix the upper bound
\begin{equation}\label{eq:OfflineReward} 
R:= \sum_{t=1}^T \sum_{i=1}^I \bE[r_{it}p_{it}(S_t)|S_1]
\end{equation}  
on $V^\OFF$. }
We will transform problem data progressively, each time making the expected total reward of $\hpolicy$ smaller on this instance.  Then we will show that at some point, the expected total reward of $\hpolicy$ is easily bounded below by a constant.

Conditioned on $S_t$ and the event that the resource has not been sold by the beginning of period $t$, let $V^{\hpolicy}(S_t)$ be the expected reward earned from assigning it to a demand unit during periods from $t$ to $T$. 
We can express $V^{\hpolicy}(S_t)$ explicitly by the following recursion:
\begin{equation}\label{eq:costh}
V^{\hpolicy}(S_t) = \sum_{i=1}^I  p_{it}(S_t)\mathbf{1}(r_{it} \geq h(S_t)) \left(r_{it}-\bE[V^{\hpolicy}(S_{t+1})|S_t]\right) + \bE[V^{\hpolicy}(S_{t+1})|S_t],
\end{equation}
and $V^{\hpolicy}(S_{T+1})=0$.  




We will work with the tree representation of the stochastic process $S_1, S_2,\ldots, S_T$.  
We call a node $S_t$ in the tree a \emph{terminal node} if $p_{it}(S_t)>0$ for some $i$ but $p_{it'}(S_{t'})=0$ for all $i=1,2,...,I$ and all descendants $S_{t'}$ of $S_t$.


We will arrive at our bound by proving two sets of structural results for the worst-case instance of the problem.  The first set of structural results concern the reward process.
\begin{lemma}\label{lem:StructureRewards}
Assume that the given problem instance achieves the worst-case ratio $V^{\hpolicy}/V^{\OFF}$.  Then without loss of generality, 
\begin{enumerate}
\item $\sum_{t'=1}^T \sum_{i=1}^I p_{it'}(S_{t'}) = \bar \cT$ almost surely.
\item The reward is scenario dependent.  That is, demand unit $(i,t)$ has reward $r_{it}(S_t)$ at each scenario $S_t$.  
\item In each scenario $S_t$, there is at most one customer type with positive arrival probability. We thus use $p(S_t)$ and $r(S_t)$ to denote the arrival probability and the reward of that customer type.
\item $r(S_t) = h(S_t)$ or $r(S_t)= (h(S_t))^-$ \x{for all non-terminal nodes $S_t$}. 
\end{enumerate}
\end{lemma}
\proof{Proof.}
The first property is easy to see, since we can always add nodes with $0$ reward and positive arrival probabilities to paths in the tree to ensure that the property holds.  This transformation does not change \x{either $R$ or the outcome of $\hpolicy$.} 

By adding demand types if necessary, we can assume without loss of generality that the reward is scenario dependent.  That is, demand unit $(i,t)$ has reward $r_{it}(S_t)$ at each scenario $S_t$.  


We can split up each period into several periods if necessary, such that each scenario $S_t$ has at most one arrival with probability $p(S_t)$ and reward $r(S_t)$.
This change preserves \x{$R$} 
 and decreases the expected reward for $h$ according to \eqref{eq:costh}, if the rewards are chosen to be increasing with time.

In the tree representation, if there is some highest-level non-terminal node $S_t$ for which $r(S_t) > h(S_t)$ then conditioned on $S_t$, we can decrease $r(S_t)$ and scale up $r(S_{t'})$ by some factor for all scenarios $S_{t'}$ that descend from $S_t$, using the same factor for all $S_{t'}$, such that the value of
\[  \sum_{t'=t}^T \bE[r(S_{t'})p(S_{t'})|S_t] \]
is unchanged. As a result, the equality in \eqref{eq:OfflineReward} is maintained. Do this until $r(S_t)=h(S_t)$. 

We claim that this change decreases  $V^{\hpolicy}(S_t)$, hence $V^{\hpolicy}(S_1)$.   To see the claim, note that according to \eqref{eq:costh}, the change reduces the immediate reward given $S_t$ by some amount $\Delta$ and increases the future reward given $S_t$ by no more than $\Delta$.  Thus the net effect is to reduce $V^{\hpolicy}(S_t)$.    

Also, since the value of $h(S_s)$ and the rewards stay the same for every node $S_s$ preceding $S_t$,  $V^{\hpolicy}(S_1)$, is reduced.

In the tree representation, if there is some highest-level non-terminal node $S_t$ for which $r(S_t) < h(S_t)$ then conditioned on $S_t$, we can increase $r(S_t)$ by a small amount and scale down $r(S_{t'})$ by some factor for all scenarios $S_{t'}$ descending from $S_t$, using the same factor for all $S_{t'}$, such that the equality in \eqref{eq:OfflineReward}  is maintained and all rewards remain non-negative.  Do this until $r(S_t)= (h(S_t))^-$, where $(h(S_t))^-$ denotes a value infinitessimally smaller than $h(S_t)$. It is easy to see that  this change decreases  $V^{\hpolicy}(S_t)$.  Hence $V^{\hpolicy}(S_1)$ is decreased as we argued just above. Repeat the previous transformations until at all non-terminal nodes $S_t$, we have  $r(S_t) = h(S_t)$ or $r(S_t)= (h(S_t))^-$.
\halmos
\endproof

Lemma \ref{lem:StructureRewards} implies that $\cT(S_T) = \bar \cT$ in the worst-case instance. We will assume for the rest of the subsection that our worst-case data has the structure imposed by Lemma \ref{lem:StructureRewards}.  Therefore, for the rest of this subsection, we write the threshold function $h$ in the following alternative way
\begin{equation}\label{eq:h}
h(S_t) = \bE\!\!\left[\frac{\sum_{t'=t+1}^T \sum_{i=1}^I r_{it'}p_{it'}(S_{t'})}{1+\sum_{t'=t+1}^T \sum_{i=1}^I p_{it'}(S_{t'})}\ \big|\ S_t\right].
\end{equation}

Our second set of structural results concern the arrival probabilities.
\begin{lemma}\label{lem:StructureArrivals}  Let $\bar \cT=1$.
Assume that the data achieves the worst-case ratio \x{$V^{\hpolicy}/R$}. 
 Without loss of generality, for $T\geq 2$, the followings hold:
\begin{enumerate}
\item There is a unique path $S_1, S_2, \ldots, S_T$ with positive arrival probabilities. 
\item At every node $S_t$, $t=1,\ldots,T-1$, $r(S_t)=h(S_T)=\bE[r(S_T)p(S_T)|S_t]$;
\item $p(S_t)=0$, $t=2,\ldots,T$;
\item $V^{\hpolicy}(S_1) \geq\frac{R}{2}$.
\end{enumerate}
\end{lemma}
\proof{Proof.}
We will prove the theorem by induction on $T$.  First, we prove the result for $T=2$. By the tree reward simplifications, 
$$r(S_1)\approx h(S_1)=\sum_{u \in\mathcal{S}_2} \bP(S_2 =u | S_1) \frac{r(u)p(u)}{1+p(u)} \leq \sum_{u \in\mathcal{S}_2}\bP(S_2 =u | S_1) r(u)p(u) = \bE[V^{\hpolicy}(S_2) | S_1].$$
Therefore $r(S_1)=h(S_1)$ to make $V^{\hpolicy}(S_1)$ as small as possible. Define
\begin{eqnarray*}
c_2&=&\sum_{u \in \mathcal{S}_2} \bP[S_2=u]p(u),\\
R_2&=&\sum_{u \in \mathcal{S}_2} \bP[S_2=u]p(u)r(u),\\
c &=&c_2 + p(S_1).
\end{eqnarray*} 
\x{Notice that by definition of $R$ we have
\[ R = R_2 +p(S_1)r(S_1).\]}

Fix $R$, $R_2$, and $c$.  Consider what happens when we scale down $c_2$ by a factor $\alpha$, scale up $r(u)$ for leaf nodes $u$ to maintain $R_2$ constant, and scale down $r(S_1)$ to maintain $p(S_1)r(S_1)=R-R_2$ constant. We argue that we will reduce $V^{\hpolicy}(S_1)$ while keeping \x{$R$} constant.  Indeed, for $\alpha=1$, 
\begin{eqnarray*}
p(S_1)&=&\bar \cT-\alpha c_2.
\end{eqnarray*}
This implies that
\begin{eqnarray*}
V^{\hpolicy}(S_1) 
&=& p(S_1)r(S_1) + (1-p(S_1))\sum_{u} \bP[S_2=u] p(u)r(u),\\
&=& R-R_2 + (1-\bar \cT+\alpha c_2)R_2.
\end{eqnarray*}
This is where we make $p(S_1)r(S_1)=R-R_2$.
Hence,
\begin{eqnarray*}
\frac{\partial V^{\hpolicy}(S_1)}{\partial \alpha} 
=c_2 R_2\geq 0.
\end{eqnarray*}
Therefore, \x{$V^{\hpolicy}/R$} 
 is minimized when $\alpha=0$, or $p(S_2)=0$ for all $S_2\in\mathcal{S}_2$.  Therefore, the base case is proved. 

Assume the theorem holds for $T-1$.  Fix $c^+ = \sum_{t=2}^T \bP(S_t|S_1)p(S_t)$ and $R^+ = \sum_{t=2}^T \bP(S_t|S_1)p(S_t)r(S_t)$. 
Let the immediate successors of $S_1$ be $S_2^k$, $k=1,\ldots,K$.  
Let $$R^k=\bE[\sum_{t=2}^T r(S_t)p(S_t)|S_2^k],$$ $k=1,\ldots,K$.  
Since the instance that minimizes $V^{\hpolicy}(S_1)$ must minimize $\bE[V^(S_2)|S_1]$ subject to $c^+$ and $R^+$, we have
by the induction hypothesis, that
\begin{eqnarray*}
\bE[V^{\hpolicy}(S_2)|S_1] 
&\geq& \sum_{1}^K \bP(S_2^k)\frac{R^k}{2}\\
&=& \frac{R^+}{2}.
\end{eqnarray*}
This lower bound is attained when $K=1$, $\sum_{t=2}^T \bP(S_t|S^K_2)p(S_t)=1$, $p(S_t)=0$ for all $t=3,\ldots,T$, and $\bP(S^K_2|S_1)=c-p(S_1)$.

By the induction hypothesis, $V^{\hpolicy}(S_2)=\bE[p(S_T)r(S_T)|S_2]=h(S_1)$. 
We also know that $r(S_1) \approx h(S_1)$ by the tree reward simplifications, we conclude that $r(S_1)=h(S_1)$, since the impact on $V^{\hpolicy}(S_1)$ is the same in either case.
Thus, 
\begin{eqnarray*}
V^{\hpolicy}(S_1) 
&=& \bE[p(S_T)r(S_T)|S_1].
\end{eqnarray*}
We know that
\begin{eqnarray*}
R
&=&p(S_T)r(S_T)(\bP(S_T|S_1)+\sum_{s=1}^{T-1}\bP(S_s|S_1)\bP(S_T|S_s)p(S_s))\\
&=&p(S_T)r(S_T)(\bP(S_T|S_1)+\bP(S^K_2|S_1)\bP(S_T|S^K_2)p(S^K_2)+\bP(S_T|S_1)p(S_1))\\
&=&p(S_T)r(S_T)\bP(S_T|S_1)(1+p(S^K_2)+p(S_1)).
\end{eqnarray*}
This implies that 
\begin{eqnarray*}
V^{\hpolicy}(S_1) 
&=&  p(S_T)r(S_T)\bP(S_T|S_1)\\
&=& \frac{R}{1+p(S^K_2)+p(S_1)}\\
&\geq&\frac{R}{2},
\end{eqnarray*}
with the lower bound being realizable when $p(S^K_2)=0$ and $p(S_1)=1$.  By induction, the lemma holds for all $T$.
\halmos
\endproof

Lemmas \ref{lem:StructureRewards} and \ref{lem:StructureArrivals} \x{and Proposition \ref{prop:prophetUpperBound} combine to give us the competitive ratio of $\hpolicy$} directly: 
\x{
\begin{theorem} \label{thm:admission}
For $\bar\cT = 1$, we have
\begin{equation}  \label{eq:ratio}
V^{\hpolicy}(S_1)  \geq  \frac{R}{2} = \frac{1}{2} \sum_{t=1}^T \sum_{i=1}^I \bE[r_{it}p_{it}(S_t)|S_1] \geq \frac{1}{2}\bE[ V^\OFF] .
\end{equation}
\end{theorem}
}

\subsubsection{Reductive martingale method for proving performance bound.}
In this section, we provide an alternative, reductive proof for the performance bound of $\hpolicy$ that works for a more general case, when \x{$\bar \cT >0$}.

Define $\tau \in [T+1]$ as the random period in which the resource is sold to a customer under $\hpolicy$. If the resource is not sold at the end of the last period $T$, we set $\tau = T+1$. In this way, $\tau$ is a stopping time bounded from above by $T+1$. We define $p_{i,T+1}(S_{T+1}) = 0$ for all $i \in [I]$.

Define a stochastic process $\{Z(S_t)\}_{t=0,1,\ldots,T+1}$ as 
\[ Z(S_t) = h(S_t) + \sum_{t'=1}^{t} \sum_{i=1}^I  p_{it'}(S_{t'})(r_{it'} - h(S_{t'}))^+.\]

\begin{proposition}\label{prop:ZReward}
$\bE[V^\hpolicy] = \bE[Z(S_\tau)]$.
\end{proposition}
\begin{proof}{Proof.}
Recall that in any period $t \in [T]$, at most one customer can arrive, i.e., $\sum_{i=1}^I X_{it} \leq 1$ with probability one.

For any $t \in [T]$, conditioned on $\tau = t$, i.e., $\hpolicy$ sells the resource in period $t$, the following two conditions must hold:
\begin{enumerate}
\item Exactly one customer arrives in period $t$, i.e., $\sum_{i=1}^I X_{it} = 1$.
\item The type $i$ of the customer who arrives in period $t$ must satisfy the threshold condition $r_{it} \geq h(S_t)$ (so that $\hpolicy$ sells the resource), or more precisely, $\sum_{i=1}^I X_{it} (r_{it} - h(S_t)) \geq 0$.
\end{enumerate}
Altogether, using the fact that $X_{it}$'s are indicators, we can obtain
\begin{equation}\label{eq:ZRewardProof1}
\sum_{i=1}^I X_{it} (r_{it} -h(S_t)) = \sum_{i=1}^I X_{it} (r_{it} -h(S_t))^+.
\end{equation}

The expected reward of $\hpolicy$ is
\begin{align*}
\bE[V^\hpolicy] = & \bE\!\left[\sum_{t=1}^T \sum_{i=1}^I X_{it} r_{it}  \mathbf{1}(\tau = t)\right]\\
=& \bE\!\left[\sum_{t=1}^T \left(\sum_{i=1}^I X_{it} (r_{it} - h(S_t)) +\sum_{i=1}^I X_{it} h(S_t)\right)  \mathbf{1}(\tau = t)\right]\\
\overset\da=&  \bE\!\left[\sum_{t=1}^T \left(\sum_{i=1}^I X_{it} (r_{it} - h(S_t)) + h(S_t)\right)  \mathbf{1}(\tau = t)\right]\\
\overset\db=&  \bE\!\left[\sum_{t=1}^T \left(\sum_{i=1}^I X_{it} (r_{it} - h(S_t))^+ + h(S_t)\right)  \mathbf{1}(\tau = t)\right]\\
=&  \bE\!\left[\sum_{t=1}^T \sum_{i=1}^I X_{it} (r_{it} - h(S_t))^+  \mathbf{1}(\tau = t)\right] + \bE\left[ \sum_{t=1}^T h(S_t)  \bI(\tau = t)\right]\\
\overset\dc=&  \bE\!\left[\sum_{t=1}^T \sum_{i=1}^I X_{it} (r_{it} - h(S_t))^+  \mathbf{1}(\tau = t)\right] + \bE[h(S_\tau)].
\end{align*}
Above, $\da$ is because $\sum_{i \in [I]} X_{it} = 1$ conditioned on $\tau = t  \in  [T]$; $\db$ is by equation \eqref{eq:ZRewardProof1}; $\dc$ is because the definition of $h(\cdot)$ naturally gives $h(S_{T+1}) = 0$.

If $\tau > t$, i.e., $\hpolicy$ does not sell the resource in period $1,2,\ldots,t$, then any customer who arrives in period $t$ must not satisfy the threshold condition. Thus, conditioned on $\tau > t$, we must have $\sum_{i=1}^I X_{it} (r_{it} - h(S_t))^+=0$. Consequently,
\[ \sum_{i=1}^I X_{it} (r_{it} - h(S_t))^+  \mathbf{1}(\tau > t) = 0\]
\[\Longrightarrow \sum_{i=1}^I X_{it} (r_{it} - h(S_t))^+  \mathbf{1}(\tau = t) = \sum_{i=1}^I X_{it} (r_{it} - h(S_t))^+  \mathbf{1}(\tau \geq t).\]

Then the expected reward can be further written as
\begin{align*}
\bE[V^\hpolicy] =& \bE\!\left[\sum_{t=1}^T \sum_{i=1}^I X_{it} (r_{it} - h(S_t))^+  \mathbf{1}(\tau = t)\right] + \bE[h(S_\tau)]\\
= & \bE\!\left[\sum_{t=1}^T \sum_{i=1}^I X_{it} (r_{it} - h(S_t))^+  \mathbf{1}(\tau \geq t)\right] + \bE[h(S_\tau)]\\
= & \bE\!\left[\sum_{t=1}^T  \bE\!\left[\sum_{i=1}^I X_{it}(r_{it} - h(S_t))^+  \mathbf{1}(\tau \geq t)\big| S_t, \{X_{i't'}\}_{i'=1,2,...,I; t'=1,2,...,t-1}\right] \right] + \bE[h(S_\tau)]\\
= & \bE\!\left[\sum_{t=1}^T  \bE\!\left[\sum_{i=1}^I X_{it}(r_{it} - h(S_t))^+ \big| S_t, \{X_{i't'}\}_{i'=1,2,...,I; t'=1,2,...,t-1}\right]  \mathbf{1}(\tau \geq t)\right] + \bE[h(S_\tau)]\\
& \text{(the event $\tau \geq t$ depends only on the information from periods $1$ to $t-1$)}\\
= & \bE\!\left[\sum_{t=1}^T  \bE\!\left[\sum_{i=1}^I X_{it}(r_{it} - h(S_t))^+ | S_t\right]  \mathbf{1}(\tau \geq t)\right] + \bE[h(S_\tau)].
\end{align*}

Finally, we use $p_{it}(S_t) = \bE[X_{it}|S_t]$ and $p_{i,T+1}(S_{T+1}) = 0$ to obtain
\begin{align*}
\bE[V^\hpolicy]= & \bE\!\left[\sum_{t=1}^T  \bE\!\left[\sum_{i=1}^I X_{it}(r_{it} - h(S_t))^+ | S_t\right]  \mathbf{1}(\tau \geq t)\right] + \bE[h(S_\tau)]\\
= & \bE\!\left[\sum_{t=1}^T \sum_{i=1}^I p_{it}(S_t)(r_{it} - h(S_t))^+  \mathbf{1}(\tau \geq t)\right] + \bE[h(S_\tau)]\\
= & \bE\!\left[\sum_{t=1}^\tau \sum_{i=1}^I p_{it}(S_t)(r_{it} - h(S_t))^+ \right] + \bE[h(S_\tau)]\\
= & \bE[Z(S_\tau)].
\end{align*}
\halmos
\end{proof}

\begin{lemma} \label{lm:ratio}
For any $b\geq 1$, $a \geq 0$, $r_1,...,r_n \geq 0$ and $p_1,...,p_n \geq 0$,
\[ \frac{a + \sum_{i=1}^n p_i r_i}{b + \sum_{i=1}^n p_i} \leq \frac{a}{b}  + \sum_{i=1}^n p_i (r_i - \frac{a}{b})^+.\]
\end{lemma}
\begin{proof}{Proof.}
Let $I \subseteq \{1,2,...,n\}$ be the set such that $r_i \geq a/b$ for all $i \in I$.
\begin{align*}
& \frac{a}{b}  + \sum_{i=1}^n p_i (r_i - \frac{a}{b})^+\\
= & \frac{a}{b}  + \sum_{i \in I} p_i (r_i - \frac{a}{b})\\
= & \frac{\left( \frac{a}{b}  + \sum_{i \in I} p_i (r_i - \frac{a}{b})\right) ( b + \sum_{i \in I} p_i) }{ b + \sum_{i \in I} p_i}\\
= & \frac{ a + \sum_{i \in I} p_i r_i  + (b-1+\sum_{i \in I} p_i)( \sum_{i \in I}p_i (r_i - a/b))}{b + \sum_{i \in I} p_i}\\
\geq & \frac{ a + \sum_{i \in I} p_i r_i }{b + \sum_{i \in I} p_i}\\
\geq & \frac{ a + \sum_{i =1}^n p_i r_i }{b + \sum_{i=1}^n p_i}.
\end{align*}
The last inequality follows from the fact that for any $j \not\in I$,
\[r_j < \frac{a}{b} \leq \frac{ a + \sum_{i \in I} p_i r_i }{b + \sum_{i\in I} p_i}.\]

\halmos
\end{proof}

\begin{proposition} \label{prop:submartingale}
The process $\{Z(S_t)\}_{t\geq 0}$ is a sub-martingale with respect to $S_t$.
\end{proposition}
\begin{proof}{Proof.}
For any $t \geq 1$, by definition of $Z(S_t)$ and $Z(S_{t-1})$, we can obtain
\begin{align*}
& \bE[Z(S_{t}) | S_{t-1}]\\
= & \bE[h(S_t) + \sum_{t'=1}^{t} \sum_{i=1}^I  p_{it'}(S_{t'})(r_{it'} - h(S_{t'}))^+ | S_{t-1}]\\
=& \bE[h(S_t) - h(S_{t-1}) + \sum_{i=1}^Ip_{it}(S_{t})(r_{it} - h(S_{t}))^+ | S_{t-1}] + Z(S_{t-1}).
\end{align*}
It suffices to prove that in expectation, $ h(S_{t-1})\leq  h(S_t) + \sum_{i=1}^I p_{it}(S_{t})(r_{it} - h(S_{t}))^+ $. We can derive
\begin{align*}
& h(S_{t-1}) \\
= & \bE\!\!\left[\frac{\sum_{t'=t}^T \sum_{i=1}^I r_{it'}p_{it'}(S_{t'})}{1 + \bar \cT -\sum_{t'=1}^{t-1} \sum_{i=1}^I p_{it'}(S_{t'})}\ \big|\ S_{t-1}\right]\\
= & \bE\!\!\left[ \frac{\sum_{t'=t+1}^T \sum_{i=1}^I r_{it'}p_{it'}(S_{t'}) + \sum_{i=1}^I r_{it}p_{it}(S_{t}) }{1 + \bar \cT-\sum_{t'=1}^{t} \sum_{i=1}^I p_{it'}(S_{t'})+ \sum_{i=1}^I p_{it}(S_{t})} \ \big|\ S_{t-1}\right]\\
= & \bE\!\!\left[ \bE\!\!\left[ \frac{\sum_{t'=t+1}^T \sum_{i=1}^I r_{it'}p_{it'}(S_{t'}) + \sum_{i=1}^I r_{it}p_{it}(S_{t}) }{1 + \bar \cT-\sum_{t'=1}^{t} \sum_{i=1}^I p_{it'}(S_{t'})+ \sum_{i=1}^I p_{it}(S_{t})} \big| S_t\right] \ \big|\ S_{t-1}\right]\\
= & \bE\!\!\left[ \frac{\bE\!\!\left[ \sum_{t'=t+1}^T \sum_{i=1}^I r_{it'}p_{it'}(S_{t'}) | S_t \right] + \sum_{i=1}^I r_{it}p_{it}(S_{t}) }{1 + \bar \cT-\sum_{t'=1}^{t} \sum_{i=1}^I p_{it'}(S_{t'})+ \sum_{i=1}^I p_{it}(S_{t})}  \ \big|\ S_{t-1}\right]\\
\leq & \bE\!\!\left[ \frac{\bE\!\!\left[ \sum_{t'=t+1}^T \sum_{i=1}^I r_{it'}p_{it'}(S_{t'}) | S_t \right]}{1 + \bar \cT-\sum_{t'=1}^{t} \sum_{i=1}^I p_{it'}(S_{t'})} + \sum_{i=1}^I p_{it}(S_t)\left(r_{it} -\frac{\bE\!\!\left[ \sum_{t'=t+1}^T \sum_{i=1}^I r_{it'}p_{it'}(S_{t'}) | S_t \right] }{1 + \bar \cT-\sum_{t'=1}^{t} \sum_{i=1}^I p_{it'}(S_{t'})} \right)^+\ \big|\ S_{t-1}\right]\\
= & \bE\!\!\left[ \bE\!\!\left[\frac{ \sum_{t'=t+1}^T \sum_{i=1}^I r_{it'}p_{it'}(S_{t'}) }{1 + \bar \cT-\sum_{t'=1}^{t} \sum_{i=1}^I p_{it'}(S_{t'})}| S_t \right] + \sum_{i=1}^I p_{it}(S_t)\left(r_{it}-\bE\!\!\left[\frac{ \sum_{t'=t+1}^T \sum_{i=1}^I r_{it'}p_{it'}(S_{t'})  }{1 + \bar \cT-\sum_{t'=1}^{t} \sum_{i=1}^I p_{it'}(S_{t'})}| S_t \right] \right)^+\ \big|\ S_{t-1}\right]\\
= & \bE\!\!\left[ h(S_t) + \sum_{i=1}^I p_{it}(S_t)\left(r_{it} -h(S_t) \right)^+\ \big|\ S_{t-1}\right],
\end{align*}
where the inequality follows from Lemma \ref{lm:ratio} and the fact that $\bar \cT$ is an upper bound on the expected total number of arrivals on any sample path: 
\[\bar \cT \geq \sum_{t'=1}^{t} \sum_{i=1}^I p_{it'}(S_{t'}) \Longrightarrow 1 + \bar \cT-\sum_{t'=1}^{t} \sum_{i=1}^I p_{it'}(S_{t'}) \geq 1.\]
\halmos
\end{proof}

With Propositions \ref{prop:ZReward} and \ref{prop:submartingale} established, we apply the optional stopping theorem, to obtain our main result, namely the prophet inequality under correlated arrival probabilities:
\begin{theorem}\label{thm:prophet}
\[ \bE[V^\hpolicy] \geq \bE[h(S_0)] = \frac{\bE[\sum_{t=1}^T \sum_{i=1}^I r_{it}p_{it}(S_{t})]}{1 + \bar \cT} \geq \frac{\bE[V^\OFF]}{1 + \bar \cT}. \]
\end{theorem}
\begin{proof}{Proof.}
\[ \bE[V^\hpolicy] = \bE[Z(S_\tau)] \geq \bE[Z(S_0)] =  \bE[h(S_0)] = \frac{\bE[\sum_{t=1}^T \sum_{i=1}^I r_{it}p_{it}(S_{t})]}{1 + \bar \cT} \geq \frac{\bE[V^\OFF]}{1 + \bar \cT},\] 
where the first inequality is by the optimal stopping theorem for sub-martingales, and the last equality is given by \eqref{eq:hS0}.

\halmos
\end{proof}

\section{Application to Two-sided Matching Problems}

In this section, we describe how our results can be applied to design algorithms for a basic matching problem in two-sided markets.  

\subsection{Model}

Again consider a finite planning horizon of $T$ periods.  There are $I$ types of demand units and $J$ types of supply units.  Both demand and supply units arrive randomly over the $T$ periods.  The demand unit of type $i \in [I]$ that arrives at time $t \in [T]$ if any, can be identified using the pair $(i,t)$.  Similarly, the supply unit of type $j \in [J]$ that arrives at time $s \in [T]$ if any, can be identified with the pair $(j,s)$.

Each demand unit $(i,t)$ has a known non-negative reward $r_{ijts}$ when matched with a supply unit $(j,s)$.  This reward can capture how far apart the units are in time and how compatible their respective types are. If types $i$ and $j$ are incompatible, then the reward $r_{ijts}$ could be very small or $0$.  If $i$ and $j$ are compatible, then $r_{ijts}$ can decrease with the length of the interval $[s, t]$ to capture the diminishing value of the match when the supply unit must wait for a long time for the demand unit.  

We assume that supply units can wait but demand units cannot. At the end of each period $t$, after arrivals of demand and supply units have been observed, the demand unit that arrives in period $t$, if any, must be matched immediately to an existing supply unit or rejected. 
Note that if a supply unit $(j,s)$ can only wait a finite amount of time, then we can require that $r_{ijts}=0$ for any $t$ that is sufficiently large compared to $s$.

Let $\Lambda_{it}\in\{0,1\}$ be a random indicator of whether demand unit $(i,t)$ arrives, and $M_{js}\in\{0,1\}$ a random indicator of whether supply unit $(j,s)$ arrives.  We assume that all the random indicators $\Lambda_{it}$, $\forall i \in [I], t \in [T]$, and $M_{js}$, $\forall j \in [J], s \in [T]$, are mutually independent. 

The arrival probabilities $\lambda_{it} := \bE[\Lambda_{it}]$ and $\mu_{js} := \bE[M_{js}]$ are deterministic and known to the platform a priori. To avoid trivialities in the analysis, we assume all the $\lambda_{it}$ and $\mu_{js}$ are strictly positive, but their values can be arbitrarily small.

Thus, regardless of the availability of supply units, the demand arrival processes are \emph{not} correlated a priori. However, for each particular supply unit, the best demand unit that should be assigned to it must depend on the availability of other supply units.  When there are many supply units, some of them might not even be assigned to any demand unit.  By contrast, when there are very few supply units, each of them can be matched to some demand unit.

We also assume that the time increments are sufficiently fine, so that at most one demand unit and one supply unit arrive in any period $t \in [T]$. More precisely, we require $\sum_{i \in [I]} \lambda_{it} \leq 1$ and $\sum_{j \in [J]} \mu_{js} \leq 1$ for all $ t,s \in [T]$. Also similar to \eqref{eq:XitDefinition}, the arrival events can be defined as
\begin{equation}\label{eq:LambdaMuDefinition}
\Lambda_{it} = \bI\{u_t \in [\sum_{k=1}^{i-1} \lambda_{kt}, \sum_{k=1}^i \lambda_{kt})\}, \quad M_{js} = \bI\{v_s \in [\sum_{k=1}^{j-1} \mu_{ks}, \sum_{k=1}^j \mu_{ks})\},
\end{equation}
where $u_1,\ldots,u_T$ and $v_1,\ldots,v_T$ are mutually independent $[0,1]$ uniform random variables.

In any period $t$, the platform first observes $\Lambda_{1t},\ldots,\Lambda_{It}$ and $M_{1t},\ldots, M_{Jt}$. Then, if there is any arriving demand unit in period $t$, the platform uses an online algorithm to make an assignment decision. In other words, a demand unit can be matched to any supply unit arriving in the same period or earlier.

The objective of the problem is to match demand and supply units in an online manner to maximize the expected total reward earned over the horizon.  We do not allow fractional matchings.  That is, each demand unit must be matched in whole to a supply unit.

\subsection{Offline Algorithm and Its Upper Bound}

An optimal offline algorithm $\OFF$ can see the arrivals of all the demand and supply units $(\Lambda,M)$ at the beginning of period 1. Given $(\Lambda,M)$, The maximum offline reward $V^\OFF(\Lambda,M)$ is equal to the value of the following maximum-weight matching problem.
\begin{align}
\begin{split}\label{eq:LP1}
V^\OFF(\Lambda,M) =& \max_{x_{ijts}(\Lambda,M), i \in [I]; j \in [J]; t,s \in [T]} \quad \sum_{i,j,t,s} x_{ijts}(\Lambda,M) r_{ijts}\\
\text{s.t.} & \sum_{i,t} x_{ijts}(\Lambda,M) \leq M_{js}, \quad \forall j \in [J]; s \in [T],\\
& \sum_{j,s} x_{ijts}(\Lambda,M) \leq \Lambda_{it}, \quad \forall i \in [I]; t \in [T],\\
& x_{ijts}(\Lambda,M) \leq \Lambda_{it}M_{js}, \quad \forall i \in [I]; j \in [J]; t,s \in [T],\\
 &  x_{ijts}(\Lambda,M) = 0, \quad \forall i \in [I]; j \in [J]; t \in [T]; s = t+1,\ldots, T,\\
& x_{ijts}(\Lambda,M) \geq 0, \quad \forall i \in [I]; j \in [J]; t,s \in [T].
\end{split}
\end{align}
In the above LP, the  variable $x_{ijts}(\Lambda,M)$ encapsulates the probability that both demand unit $(i,t)$ and supply unit $(j,s)$ arrive \emph{and} $(i,t)$ is assigned to  $(j,s)$. The fourth constraint requires that a demand unit in period $t$ cannot be matched to any supply unit arriving later than $t$.

The competitive ratio $c$ of an online algorithm $\ON$ is similarly defined as 
\[c = \bE[V^\ON]/\bE[V^\OFF(\Lambda,M)],\] 
where $V^\ON$ is the total reward of the online algorithm, and the expectation is taken over $(\Lambda, M)$.

Note that \eqref{eq:LP1} cannot be solved without a priori access to the realizations of $(\Lambda,M)$. Thus, we are interested in finding an upper bound on the expected optimal offline reward $\bE[V^\OFF(\Lambda,M)]$ when we \emph{do not} have such a priori access.  

The following LP solves for the total probabilities $x_{ijts}$ of having demand unit $(i,t)$ arrived and assigned to $(j,s)$.
\begin{align}
\begin{split}\label{eq:LP3}
\max_{x_{ijts}} & \,\, \sum_{i,j,t,s} x_{ijts} r_{ijts}\\
\mbox{s.t. }& \sum_{i,t} x_{ijts} \leq \mu_{js}, \quad \forall j \in [J]; s \in [T],\\
& \sum_{j,s} x_{ijts} \leq \lambda_{it}, \quad \forall i \in [I]; t \in [T],\\
& x_{ijts} \leq \lambda_{it}\mu_{js}, \quad \forall i \in [I]; j \in [J]; t,s \in [T],\\
& x_{ijts} = 0, \quad \forall i \in [I]; j \in [J]; t \in [T]; s = t+1,\ldots, T,\\
& x_{ijts} \geq 0, \quad \forall i \in [I]; j \in [J]; t,s \in [T].
\end{split}
\end{align}
The constraints above are derived from those of \eqref{eq:LP1}.

\begin{theorem}\label{thm:upperbound}
The optimal objective value of \eqref{eq:LP3} is an upper bound on $\bE[V^\OFF(\Lambda,M)]$.
\end{theorem}
\proof{Proof.}
Let $x^*_{ijts}(\Lambda,M)$ be an optimal solution to LP \eqref{eq:LP1}. Define a solution to LP \eqref{eq:LP3} as
\[ \bar x_{ijts} := \bE[x^*_{ijts}(\Lambda,M)].\]

Since $\sum_{i,t} x^*_{ijts}(\Lambda,M) \leq M_{js}$, $\sum_{j,s} x^*_{ijts}(\Lambda,M) \leq \Lambda_{it}$, and $x^*_{ijts}(\Lambda,M) \leq \Lambda_{it}M_{js}$ are required in \eqref{eq:LP1}, we must have 
\begin{align*}
&\sum_{i,t} \bar x_{ijts} =\sum_{i,t} \bE[x^*_{ijts}(\Lambda,M)] \leq \bE[M_{js}] = \mu_{js},\\
&\sum_{j,s} \bar x_{ijts} =\sum_{j,s} \bE[x^*_{ijts}(\Lambda,M)] \leq \bE[\Lambda_{it}] = \lambda_{it},\\
&\bar x_{ijts} = \bE[x^*_{ijts}(\Lambda,M)] \leq \bE[\Lambda_{it}M_{js}] = \lambda_{it} \mu_{js}.
\end{align*}

Also, $x^*_{ijts}(\Lambda,M) = 0$ for all $s > t$ implies $\bar x_{ijts} = 0$ for all $s > t$.

Thus, $\bar x_{ijts}$ is a feasible solution to LP (\ref{eq:LP3}). It follows that the optimal value of LP (\ref{eq:LP3}) is an upper bound on
\[ \sum_{i,j,t,s} \bar x_{ijts} r_{ijts} =  \sum_{i,j,t,s} \bE[x^*_{ijts}(\Lambda,M)] r_{ijts}  = \bE[V^\OFF(\Lambda,M)]. \]

\halmos
\endproof

\subsection{Online Algorithm for Two-Sided Matching}
In this section, we describe and analyze a matching algorithm.  

The algorithm is composed of two simpler sub-routines, a Separation Subroutine and an Admission Subroutine.
The Separation Subroutine randomly samples a supply unit for each incoming demand unit. This sampling splits the demand arrivals into separate arrival streams, each coming to a separate supply unit.  Subsequently, for each supply unit independently, the Admission Subroutine uses the algorithm in Section \ref{sec:prophetAlg} to control the matching of at most one among all incoming demand units to it.  

Let $\mathcal{S}_t := \{0,1\}^{J\times t}$. Define $S_t \in \mathcal{S}_t$ as the information set that records the arrivals of supply units up to period $t$. That is, 
\[S_t= (\{M_{j1}\}_{j=1,2,...,J}, \{M_{j2}\}_{j=1,2,...,J},..., \{M_{jt}\}_{j=1,2,...,J}).\]
For convenience, let $S_0$ be a dummy constant. In our analysis, $S_T = M$ is the \emph{sample path} of scenarios defined in Section \ref{sec:prophetModel}.

The matching algorithm first needs to compute an optimal solution $x^*$ to LP \eqref{eq:LP3}. Then, the Separation Subroutine calculates a probability
\begin{equation}\label{eq:pijts}
p_{ijts}(S_t) :=  \frac{ \min(\lambda_{it}, \sum_{j'=1}^J \sum_{s'=1}^t M_{j's'}\frac{x^*_{ij'ts'}}{\mu_{j's'}}    )}{ \sum_{j'=1}^J \sum_{s'=1}^t M_{j's'}\frac{x^*_{ij'ts'}}{\mu_{j's'}} } \cdot M_{js}  \frac{x^*_{ijts}}{\mu_{js}}
\end{equation}
of choosing $(j,s)$ as a candidate supply unit to be matched to $(i,t)$. 
Note that if $s > t$, then we must have $x^*_{ijts} = 0$ (see LP \eqref{eq:LP3}) and thus $p_{ijts}(S_t) = 0$. That is, our algorithm never tries to match a demand unit in period $t$ to a supply unit arriving later than $t$.

When applying the prophet inequality theory developed in previous sections, we will fix a supply unit $(j,s)$, and think of $(p_{ijts}(S_t))_{i,t}$ as the probabilities that demand units ``arrive'' at $(j,s)$. We first establish some important properties regarding the arrival probabilities $p_{ijts}(\cdot)$.

\begin{proposition}\label{prop:pijts}
\begin{enumerate}
\item[]
\item[1.] $\sum_{i=1}^I \sum_{t=1}^T p_{ijts}(S_t) \leq 1$, for all $j \in [J]$ and $s \in [T]$.
\item[2.] $\sum_{j=1}^J\sum_{s=1}^t p_{ijts}(S_t) \leq \lambda_{it}$, for all $i \in [I]$ and $t \in [T]$.
\end{enumerate}
\end{proposition}
\begin{proof}{Proof.}
\begin{align*}
\sum_{i=1}^I \sum_{t=1}^T p_{ijts}(S_t) & = \sum_{i=1}^I \sum_{t=1}^T \frac{ \min(\lambda_{it}, \sum_{j'=1}^J \sum_{s'=1}^t M_{j's'}\frac{x^*_{ij'ts'}}{\mu_{j's'}}    )}{ \sum_{j'=1}^J \sum_{s'=1}^t M_{j's'}\frac{x^*_{ij'ts'}}{\mu_{j's'}} } \cdot M_{js}  \frac{x^*_{ijts}}{\mu_{js}}\\
& \leq M_{js} \sum_{i=1}^I \sum_{t=1}^T \frac{x^*_{ijts}}{\mu_{js}}\\
& \leq M_{js}\\
& \leq 1,
\end{align*}
where the second inequality is given by the first constraint of LP (\ref{eq:LP3}).

We can then derive
\begin{align*}
& \sum_{j=1}^J\sum_{s=1}^t p_{ijts}(S_t) \\
= & \sum_{j=1}^J\sum_{s=1}^t \frac{ \min(\lambda_{it}, \sum_{j'=1}^J \sum_{s'=1}^t M_{j's'}\frac{x^*_{ij'ts'}}{\mu_{j's'}}    )}{ \sum_{j'=1}^J \sum_{s'=1}^t M_{j's'}\frac{x^*_{ij'ts'}}{\mu_{j's'}} } \cdot M_{js}  \frac{x^*_{ijts}}{\mu_{js}} \\
= & \frac{ \min(\lambda_{it}, \sum_{j'=1}^J \sum_{s'=1}^t M_{j's'}\frac{x^*_{ij'ts'}}{\mu_{j's'}}    )}{ \sum_{j'=1}^J \sum_{s'=1}^t M_{j's'}\frac{x^*_{ij'ts'}}{\mu_{j's'}} } \cdot  \sum_{j=1}^J\sum_{s=1}^t M_{js}  \frac{x^*_{ijts}}{\mu_{js}}\\
= &\min(\lambda_{it}, \sum_{j'=1}^J \sum_{s'=1}^t M_{j's'}\frac{x^*_{ij'ts'}}{\mu_{j's'}}    )\\
\leq & \lambda_{it}.
\end{align*}

\halmos
\end{proof}

\noindent {\bf Online Matching Algorithm:}
\begin{itemize}
\item (Initialization) Solve \eqref{eq:LP3} for an optimal solution $x^*$.  
\item Upon an arrival of a demand unit $(i,t)$ in period $t$:
\begin{enumerate}
\item (Separation Subroutine) Randomly pick a supply unit $(j,s)$ with probability $p_{ijts}(S_t)/\lambda_{it}$, for all $j \in [J]$, $s \in [t]$ (recall that we assume $\lambda_{it}$ to be strictly positive). Notice that by definition of $p_{ijts}(\cdot)$, only those supply units that have arrived (i.e., satisfy $M_{js}=1$ and $s \leq t$) can have a positive probability to be picked. 

Also, since Proposition \ref{prop:pijts} gives $\sum_{j,s}  p_{ijts}(S_t) / \lambda_{it} \leq 1$, if the inequality is strict (i.e., $\sum_{j,s} p_{ijts}(S_t) / \lambda_{it} < 1$), then it is possible that no supply unit is picked. In such a case, reject the demand unit directly.

\item (Admission Subroutine) Let $X_{ijts}$ be the indicator of whether demand unit $(i,t)$ arrives \emph{and} the Separation Routine picks supply unit $(j,s)$. We have
\[ \bE[ X_{ijts} | S_t] = \bP(\Lambda_{it}=1) \cdot p_{ijts}(S_t)/\lambda_{it} = \lambda_{it} \cdot p_{ijts}(S_t)/\lambda_{it} = p_{ijts}(S_t).\]

For the supply unit $(j,s)$ picked by the Separation Subroutine (i.e., $X_{ijts}=1$), we apply algorithm $\hpolicy$ by viewing $(p_{ijts}(S_t))_{i,t}$ as the sequence of correlated arrival probabilities. Specifically, match $(i,t)$ to $(j,s)$ if $(j,s)$ is still available and $r_{ijts} \geq h_{js}(S_t)$, where
\[ h_{js}(S_t) := \bE\!\!\left[\frac{\sum_{t'=t+1}^T \sum_{i=1}^I r_{ijt's}p_{ijt's}(S_{t'})}{2 -\sum_{t'=1}^t \sum_{i=1}^I p_{ijt's}(S_{t'})}\ \big|\ S_t\right].\]
Notice that we have chosen $\bar \cT = 1$ (see \eqref{eq:hnew}) for this two-sided online matching problem. This is because the first property of Proposition \ref{prop:pijts} guarantees $\sum_{i=1}^I \sum_{t=1}^T p_{ijts}(S_t) \leq 1 = \bar \cT$.
\end{enumerate}
\end{itemize}

\subsection{Performance of the Online Algorithm}

We first establish an approximation bound for the Separation Subroutine, which relates the arrival probabilities $p_{ijts}(\cdot)$ to the LP upper bound \eqref{eq:LP3}.
We start with a technical lemma.
\begin{lemma}\label{lm:routinglm1}
For any $\lambda > 0$ and $x > 0$,
\[ \frac{\min(\lambda,x)}{x} \geq 1 - \frac{1}{4\lambda} x.\]
\end{lemma}
\begin{proof}{Proof.}
For $x \leq \lambda$, $\min(\lambda, x) / x = 1 \geq 1 - \frac{1}{4\lambda} x$.

For $x > \lambda$, 
\begin{align*}
& \frac{\min(\lambda,x)}{x} - (1 - \frac{1}{4\lambda} x)\\
= & \frac{\lambda}{x} - 1 + \frac{1}{4\lambda} x\\
= & \frac{\lambda}{x}  + \frac{1}{4} \cdot \frac{x}{\lambda} - 1\\
\geq & 2 \cdot \sqrt{\frac{\lambda}{x}} \cdot \frac{1}{2} \sqrt{\frac{x}{\lambda}} - 1\\
\geq & 0.
\end{align*}
\halmos
\end{proof}
Using the above lemma and the definition of $p_{ijts}(\cdot)$, we are ready to prove the proximity of $p_{ijts}(\cdot)$ relative to $x^*$.
\begin{theorem}\label{thm:pijtsBound}
$\bE[p_{ijts}(S_t)] \geq 0.5  x^*_{ijts}$.
\end{theorem}
\begin{proof}{Proof.}
Fix any $i,j,t,s$, but $M$ (and thus $S_t$) is random. For any supply unit $(j',s')$, define
\[ Y_{j's'} \equiv M_{j's'}\cdot  x^*_{ij'ts'} / \mu_{j's'}.\]

Note that $\bE[Y_{j's'}] = x^*_{ij'ts'} / \mu_{j's'} \cdot \bP(M_{j's'}=1) = x^*_{ij'ts'} / \mu_{j's'} \cdot \mu_{j's'} = x^*_{ij'ts'}$.

We can then deduce
\begin{align*}
& \bE[p_{ijts}(S_t)]\\
= & \bE[\frac{\min(\lambda_{it}, \sum_{j',s'} Y_{j's'})}{\sum_{j',s'} Y_{j's'}} \cdot Y_{js}]\\
\geq & \bE[ (1 - \frac{1}{4\lambda_{it}} \sum_{j',s'} Y_{j's'}) \cdot Y_{js}] \\
& \,\,\,\,\,\text{(by Lemma \ref{lm:routinglm1}; if $\sum_{j',s'} Y_{j's'} = 0$, we have $Y_{js} = 0$ so the inequality still holds)}\\
= & \bE[1 - \frac{1}{4\lambda_{it}} \sum_{j' \not= j,s'\not= s} Y_{j's'}] \bE[Y_{js}] - \frac{1}{4 \lambda_{it}}\bE[Y_{js}^2] \\
= & (1 - \frac{1}{4\lambda_{it}} \sum_{j'\not= j, s'\not= s}x^*_{ij'ts'})x^*_{ijts} - \frac{1}{4\lambda_{it}}  x^*_{ijts}\cdot x^*_{ijts} / \mu_{js}\\
\geq & (1 - \frac{1}{4})x^*_{ijts} - \frac{1}{4} x^*_{ijts}\\
& \,\,\,\,\,\text{(because the constraints of LP \eqref{eq:LP3} requires $\sum_{j's'}x^*_{ij'ts'} \leq \lambda_{it}$ and $x^*_{ijts} \leq \mu_{js} \lambda_{it}$)}\\
=& \frac{1}{2} x^*_{ijts}.
\end{align*}
\halmos
\end{proof}

Now, we tie together the above approximation bound with the prophet inequality established in Section \ref{sec:prophetAlg}:
\begin{theorem}  
The total reward $V^\ON$ of our matching algorithm satisfies 
\[\bE[V^\ON] \geq \frac{1}{4} \bE[V^\OFF(\Lambda,M)].\]
\end{theorem}
\proof{Proof.} 
Fix any $(j,s)$ for $j \in [J]$ and $s \in [T]$. A demand unit $(i,t)$ is matched to $(j,s)$ if and only if $X_{ijts}=1$ and $r_{ijts} \geq h_{js}(S_t)$. This is exactly the single-resource problem presented in Section \ref{sec:prophetModel}. Therefore, by Theorem \ref{thm:prophet}, the expected total reward earned from $(j,s)$ is at least (recall that we choose $\bar \cT = 1$ for this two-sided online matching problem)
\[ \bE[h_{js}(S_0)] = \frac{ \bE[ \sum_{t=1}^T \sum_{i=1}^I r_{ijts}p_{ijts}(S_{t})]}{2}.\]
We then use Theorem \ref{thm:pijtsBound} to obtain
\[ \bE[h_{js}(S_0)] = \frac{ \bE[ \sum_{t=1}^T \sum_{i=1}^I r_{ijts}p_{ijts}(S_{t})]}{2} = \frac{ \sum_{t=1}^T \sum_{i=1}^I r_{ijts}\bE[ p_{ijts}(S_{t})]}{2} \geq \frac{ \sum_{t=1}^T \sum_{i=1}^I r_{ijts} x^*_{ijts}}{4}.\]
Consequently, the total expected reward summed over all supply units is at least
\[ \sum_{j \in [J]} \sum_{s \in [T]} \bE[h_{js}(S_0)] \geq  \sum_{j \in [J]} \sum_{s \in [T]}\frac{ \sum_{t \in [T]} \sum_{i \in [I]} r_{ijts} x^*_{ijts}}{4} \geq \frac{1}{4} \bE[V^\OFF(\Lambda,M)], \]
where the final inequality follows from Theorem \ref{thm:upperbound}.
\halmos
\endproof
Note that the ratio of $1/4$ above results from a loss of a factor of $1/2$ from the solution of Prophet Inequalities, and another factor of $1/2$ from the tractable approximation to the upper-bound deterministic LP. \x{Since $1/2$ is an upper bound on the competitive ratio for the standard prophet inequality (which is a special case of our prophet inequality with correlated arrival probabilities), any improvement to the bound of the online algorithm 
must come from refining the solution to the LP.} 

\section{Numerical Studies}
\label{sec:numerical}

In this section, we conduct numerical experiments to explore the performance of our algorithms. We model our experiments on applications that match employers with freelancers for short-term projects, such as web design, art painting, and data entry. 

We assume there are $30$ employer (demand) types and $30$ worker (supply) types. We set the reward $r_{ijts}$ according to the formula
\[ r_{ijts} = s_{ij} \cdot f_{ts} \cdot g_{ij},\]
where we use $s_{ij}, f_{ts}$ and $g_{ij}$ to capture three different aspects of a matching:
\begin{itemize}
\item \textbf{Ability to accomplish tasks.} $s_{ij}$ represents the ability of workers of type $j$ to work for employers of type $i$. We randomly draw $s_{ij}$ from a normal distribution $\cal{N}(0,1)$ for each pair $(i,j)$. 
In particular, if $s_{ij} < 0$, the reward of the matching will be negative, and thus no algorithm will ever match worker type $j$ to employer type $i$.
\item \textbf{Idle time of workers.} It may be wise to limit the total time that a worker is idle in the system before being assigned a job. Thus, we set
\[ f_{ts} = 1-\alpha + \alpha e^{-(t-s)/\tau}\]
so that the reward of a matching is discounted by $\alpha$ when the idle time of the worker exceeds $\tau$.
\item \textbf{Geographical distance.} Certain freelance jobs may require a short commute distance between workers and employers. For demonstration purpose, we assume that each worker type and employer type is associated with a random zip code in Manhattan, with probability proportional to the total population in the zip code zone. Let $d(i,j)$ be the Manhattan distance between the centers of zip code zones of worker type $j$ and employer type $i$. We assume that
\[ g_{ij} = 1 - \beta + \beta e^{-d(i,j) / \omega},\]
so the reward is discounted by $\beta$ when the commute distance exceeds $\omega$. 
\end{itemize}


We consider a horizon of 60 periods. Depending on the application, one period may correspond to a day or a 10-minute span. In any period, a random number of workers may sign in to be ready to provide service. The type of a worker is uniformly drawn from all the worker types. Let $\mu(t)$ be the rate at which workers appear in the system. Similarly, when an employer arrives, the type of the employer is uniformly drawn from all the employer types. Let $\lambda(t)$ be the arrival rate of employers.

We randomly generate multiple test scenarios. In each scenario, we independently draw $\mu(t)$ and $\lambda(t)$, for every period $t$, from a uniform distribution over $[0,1]$. Given the rates $\mu(t)$ and $\lambda(t)$, we further vary other model parameters by first choosing the base case to be $\alpha = 0.5$, $\tau = 10 \text{ (periods)}$, $\beta = 0.5$, $\omega = 0.05^\circ$, and then each time varying one of these parameters.

We test the following algorithms:
\begin{itemize} 
\item ($\ON$) Our online algorithm without resource sharing.


\item ($\ON_+^1$) A variant of our online algorithm with resource sharing. When $\ON_+^1$ rejects a customer in the admission subroutine, $\ON_+^1$ offers another resource with the largest non-negative margin
\[ \bE\left[ \frac{\sum_{t'=t+1}^T \sum_{i=1}^I r_{it'}p_{it'}(S_{t'})}{2 - \sum_{t'=1}^t \sum_{i=1}^I p_{it'}(S_{t'})}  \big|\ S_t\right] - r_{ijts}.\]

\item ($\ON_+^2$) A variant of our online algorithm with resource sharing. When $\ON_+^2$ rejects a customer in the admission subroutine, $\ON_+^2$ offers another resource with the largest non-negative margin
\[ 130\% \times \bE\left[ \frac{\sum_{t'=t+1}^T \sum_{i=1}^I r_{it'}p_{it'}(S_{t'})}{2 - \sum_{t'=1}^t \sum_{i=1}^I p_{it'}(S_{t'})}  \big|\ S_t\right] - r_{ijts}.\]

\item ($\ON_+^3$) A variant of our online algorithm with resource sharing. When $\ON_+^3$ rejects a customer in the admission subroutine, $\ON_+^3$ offers another resource with the largest non-negative margin
\[ 160\% \times\bE\left[ \frac{\sum_{t'=t+1}^T \sum_{i=1}^I r_{it'}p_{it'}(S_{t'})}{2 - \sum_{t'=1}^t \sum_{i=1}^I p_{it'}(S_{t'})}  \big|\ S_t\right] - r_{ijts}.\]

\item ($\ON_+^4$) A variant of our online algorithm with resource sharing. When $\ON_+^4$ rejects a customer in the admission subroutine, $\ON_+^4$ offers another resource with the largest non-negative margin
\[ 200\% \times \bE\left[ \frac{\sum_{t'=t+1}^T \sum_{i=1}^I r_{it'}p_{it'}(S_{t'})}{2 - \sum_{t'=1}^t \sum_{i=1}^I p_{it'}(S_{t'})}  \big|\ S_t\right] - r_{ijts}.\]

\item A greedy algorithm that always offers a resource with the highest reward.
\item A bid-price heuristic based on the optimal dual prices of LP (\ref{eq:LP1})
\end{itemize}

We report numerical results in Tables \ref{tab:simulation1} to \ref{tab:simulation4}, where the performance of each algorithm is simulated using 1000 replicates. For our online algorithms, in each period we compute the threshold $h_{js}(S_t)$ by simulating 100 future sample paths. We find that, despite the $1/4$ provable ratio, the algorithm $\ON$ captures about half of the offline expected reward, and the improved algorithms $\ON_+^1$ and $\ON_+^2$ capture $65\%$ to $70\%$ of the offline expected reward. Moreover, the improved algorithms outperform the greedy and the bid-price heuristics in all scenarios. These results demonstrate the advantage of using our online algorithms as they have not only optimized performance in the worst-case scenario, but satisfactory performance on average as well.

\begin{table}
\begin{center}
\caption{Scenario 1. Performance of different algorithms relative to LP (\ref{eq:LP1}).}
\label{tab:simulation1}
\begin{tabular}{|c|c|c|c|c|c|c|c|}
\hline
& $\ON$ & Greedy & BPH & $\ON_+^1$ & $\ON_+^2$ & $\ON_+^3$ & $\ON_+^4$\\
\hline
	Base	&$	49.8\%	$&$	62.2\%	$&$	63.7\%	$&$	66.1\%	$&$	67.9\%	$&$	67.9\%	$&$	67.5\%	$\\
$	\alpha = 0	$&$	48.0\%	$&$	56.9\%	$&$	63.4\%	$&$	65.3\%	$&$	67.6\%	$&$	67.9\%	$&$	67.4\%	$\\
$	\alpha = 0.2	$&$	48.6\%	$&$	58.8\%	$&$	64.0\%	$&$	65.7\%	$&$	67.8\%	$&$	67.9\%	$&$	67.5\%	$\\
$	\alpha = 0.8	$&$	51.8\%	$&$	64.9\%	$&$	62.9\%	$&$	66.1\%	$&$	67.6\%	$&$	67.6\%	$&$	67.2\%	$\\
$	\alpha = 1	$&$	53.4\%	$&$	65.3\%	$&$	64.0\%	$&$	65.7\%	$&$	66.9\%	$&$	66.8\%	$&$	66.4\%	$\\
$	\tau = 2	$&$	50.3\%	$&$	62.8\%	$&$	65.7\%	$&$	66.9\%	$&$	68.8\%	$&$	69.0\%	$&$	68.3\%	$\\
$	\tau = 5	$&$	50.3\%	$&$	63.1\%	$&$	64.8\%	$&$	66.6\%	$&$	68.2\%	$&$	68.5\%	$&$	67.9\%	$\\
$	\tau = 20	$&$	49.3\%	$&$	60.8\%	$&$	63.2\%	$&$	65.8\%	$&$	67.5\%	$&$	67.8\%	$&$	67.4\%	$\\
$	\tau = 30	$&$	49.1\%	$&$	60.0\%	$&$	63.2\%	$&$	65.6\%	$&$	67.6\%	$&$	67.7\%	$&$	67.1\%	$\\
$	\beta = 0	$&$	49.9\%	$&$	62.6\%	$&$	63.6\%	$&$	66.1\%	$&$	68.0\%	$&$	68.1\%	$&$	67.5\%	$\\
$	\beta = 0.2	$&$	49.9\%	$&$	62.6\%	$&$	63.7\%	$&$	66.1\%	$&$	68.0\%	$&$	67.9\%	$&$	67.7\%	$\\
$	\beta = 0.8	$&$	49.9\%	$&$	60.4\%	$&$	63.2\%	$&$	65.3\%	$&$	67.0\%	$&$	67.1\%	$&$	66.6\%	$\\
$	\beta = 1	$&$	49.6\%	$&$	56.9\%	$&$	62.4\%	$&$	64.0\%	$&$	65.7\%	$&$	65.8\%	$&$	65.4\%	$\\
$	\omega = 0.005	$&$	49.7\%	$&$	61.1\%	$&$	63.0\%	$&$	65.6\%	$&$	67.4\%	$&$	67.4\%	$&$	66.9\%	$\\
$	\omega = 0.02	$&$	49.8\%	$&$	61.6\%	$&$	63.5\%	$&$	65.8\%	$&$	67.6\%	$&$	67.6\%	$&$	67.1\%	$\\
$	\omega = 0.08	$&$	49.9\%	$&$	62.5\%	$&$	63.8\%	$&$	66.2\%	$&$	68.0\%	$&$	68.1\%	$&$	67.5\%	$\\
$	\omega = 0.15	$&$	50.1\%	$&$	62.7\%	$&$	63.7\%	$&$	66.1\%	$&$	68.0\%	$&$	68.2\%	$&$	67.6\%	$\\
\hline
\end{tabular}
\end{center}
\end{table}

\begin{table}
\begin{center}
\caption{Scenario 2. Performance of different algorithms relative to LP (\ref{eq:LP1}).}
\label{tab:simulation2}
\begin{tabular}{|c|c|c|c|c|c|c|c|}
\hline
& $\ON$ & Greedy & BPH & $\ON_+^1$ & $\ON_+^2$ & $\ON_+^3$ & $\ON_+^4$\\
\hline
	Base	&$	50.3\%	$&$	65.4\%	$&$	66.6\%	$&$	67.7\%	$&$	69.6\%	$&$	69.4\%	$&$	69.2\%	$\\
$	\alpha = 0	$&$	48.2\%	$&$	61.2\%	$&$	67.3\%	$&$	67.9\%	$&$	70.1\%	$&$	70.3\%	$&$	70.0\%	$\\
$	\alpha = 0.2	$&$	49.2\%	$&$	63.1\%	$&$	67.3\%	$&$	67.8\%	$&$	69.9\%	$&$	70.1\%	$&$	69.7\%	$\\
$	\alpha = 0.8	$&$	51.7\%	$&$	66.5\%	$&$	65.1\%	$&$	67.0\%	$&$	68.7\%	$&$	68.4\%	$&$	68.2\%	$\\
$	\alpha = 1	$&$	52.5\%	$&$	66.5\%	$&$	65.2\%	$&$	66.1\%	$&$	67.7\%	$&$	67.5\%	$&$	67.1\%	$\\
$	\tau = 2	$&$	51.3\%	$&$	67.4\%	$&$	69.3\%	$&$	69.4\%	$&$	71.5\%	$&$	71.3\%	$&$	71.0\%	$\\
$	\tau = 5	$&$	50.9\%	$&$	66.6\%	$&$	67.8\%	$&$	68.3\%	$&$	70.1\%	$&$	70.0\%	$&$	69.7\%	$\\
$	\tau = 20	$&$	49.7\%	$&$	64.1\%	$&$	66.3\%	$&$	67.4\%	$&$	69.6\%	$&$	69.5\%	$&$	69.2\%	$\\
$	\tau = 30	$&$	49.4\%	$&$	63.5\%	$&$	66.3\%	$&$	67.6\%	$&$	69.6\%	$&$	69.5\%	$&$	69.4\%	$\\
$	\beta = 0	$&$	50.5\%	$&$	66.5\%	$&$	67.1\%	$&$	68.6\%	$&$	70.5\%	$&$	70.4\%	$&$	69.9\%	$\\
$	\beta = 0.2	$&$	50.5\%	$&$	66.4\%	$&$	66.9\%	$&$	68.3\%	$&$	70.3\%	$&$	70.0\%	$&$	69.8\%	$\\
$	\beta = 0.8	$&$	50.1\%	$&$	62.4\%	$&$	65.9\%	$&$	66.2\%	$&$	68.1\%	$&$	67.8\%	$&$	67.7\%	$\\
$	\beta = 1	$&$	50.4\%	$&$	58.2\%	$&$	64.8\%	$&$	64.5\%	$&$	66.5\%	$&$	66.3\%	$&$	66.1\%	$\\
$	\omega = 0.005	$&$	50.2\%	$&$	64.6\%	$&$	66.0\%	$&$	67.1\%	$&$	68.9\%	$&$	68.7\%	$&$	68.6\%	$\\
$	\omega = 0.02	$&$	50.2\%	$&$	64.9\%	$&$	66.2\%	$&$	67.3\%	$&$	69.3\%	$&$	69.0\%	$&$	68.8\%	$\\
$	\omega = 0.08	$&$	50.4\%	$&$	65.8\%	$&$	66.7\%	$&$	68.0\%	$&$	69.8\%	$&$	69.7\%	$&$	69.5\%	$\\
$	\omega = 0.15	$&$	50.3\%	$&$	66.1\%	$&$	67.0\%	$&$	68.2\%	$&$	70.2\%	$&$	70.1\%	$&$	69.8\%	$\\
\hline
\end{tabular}
\end{center}
\end{table}

\begin{table}
\begin{center}
\caption{Scenario 3. Performance of different algorithms relative to LP (\ref{eq:LP1}).}
\label{tab:simulation3}
\begin{tabular}{|c|c|c|c|c|c|c|c|}
\hline
& $\ON$ & Greedy & BPH & $\ON_+^1$ & $\ON_+^2$ & $\ON_+^3$ & $\ON_+^4$\\
\hline
	Base	&$	51.2\%	$&$	63.5\%	$&$	65.2\%	$&$	67.0\%	$&$	68.2\%	$&$	67.7\%	$&$	67.1\%	$\\
$	\alpha = 0	$&$	48.6\%	$&$	58.1\%	$&$	65.8\%	$&$	66.7\%	$&$	68.2\%	$&$	67.7\%	$&$	66.8\%	$\\
$	\alpha = 0.2	$&$	49.7\%	$&$	60.2\%	$&$	66.1\%	$&$	66.8\%	$&$	68.1\%	$&$	67.6\%	$&$	66.9\%	$\\
$	\alpha = 0.8	$&$	53.2\%	$&$	65.9\%	$&$	63.8\%	$&$	66.3\%	$&$	67.6\%	$&$	67.1\%	$&$	66.6\%	$\\
$	\alpha = 1	$&$	54.2\%	$&$	66.2\%	$&$	64.7\%	$&$	65.4\%	$&$	66.5\%	$&$	66.0\%	$&$	65.6\%	$\\
$	\tau = 2	$&$	52.1\%	$&$	64.9\%	$&$	68.1\%	$&$	68.2\%	$&$	69.5\%	$&$	69.0\%	$&$	68.5\%	$\\
$	\tau = 5	$&$	52.0\%	$&$	64.8\%	$&$	66.7\%	$&$	67.4\%	$&$	68.7\%	$&$	68.2\%	$&$	67.6\%	$\\
$	\tau = 20	$&$	50.4\%	$&$	61.9\%	$&$	64.7\%	$&$	66.6\%	$&$	68.1\%	$&$	67.4\%	$&$	66.9\%	$\\
$	\tau = 30	$&$	50.1\%	$&$	61.0\%	$&$	64.9\%	$&$	66.7\%	$&$	68.0\%	$&$	67.5\%	$&$	66.9\%	$\\
$	\beta = 0	$&$	51.0\%	$&$	64.4\%	$&$	65.3\%	$&$	67.4\%	$&$	68.9\%	$&$	68.2\%	$&$	67.6\%	$\\
$	\beta = 0.2	$&$	51.4\%	$&$	64.4\%	$&$	65.1\%	$&$	67.4\%	$&$	68.6\%	$&$	68.1\%	$&$	67.5\%	$\\
$	\beta = 0.8	$&$	50.8\%	$&$	60.5\%	$&$	64.8\%	$&$	64.9\%	$&$	66.3\%	$&$	65.8\%	$&$	65.0\%	$\\
$	\beta = 1	$&$	50.8\%	$&$	56.3\%	$&$	64.0\%	$&$	63.4\%	$&$	64.7\%	$&$	64.2\%	$&$	63.4\%	$\\
$	\omega = 0.005	$&$	50.6\%	$&$	62.8\%	$&$	65.2\%	$&$	66.4\%	$&$	67.7\%	$&$	67.2\%	$&$	66.5\%	$\\
$	\omega = 0.02	$&$	51.0\%	$&$	62.9\%	$&$	65.0\%	$&$	66.4\%	$&$	67.8\%	$&$	67.3\%	$&$	66.5\%	$\\
$	\omega = 0.08	$&$	51.4\%	$&$	63.9\%	$&$	65.2\%	$&$	67.0\%	$&$	68.3\%	$&$	67.8\%	$&$	67.4\%	$\\
$	\omega = 0.15	$&$	51.3\%	$&$	64.2\%	$&$	65.1\%	$&$	67.2\%	$&$	68.6\%	$&$	68.2\%	$&$	67.4\%	$\\
\hline
\end{tabular}
\end{center}
\end{table}

\begin{table}
\begin{center}
\caption{Scenario 4. Performance of different algorithms relative to LP (\ref{eq:LP1}).}
\label{tab:simulation4}
\begin{tabular}{|c|c|c|c|c|c|c|c|}
\hline
& $\ON$ & Greedy & BPH & $\ON_+^1$ & $\ON_+^2$ & $\ON_+^3$ & $\ON_+^4$\\
\hline
	Base	&$	50.4\%	$&$	64.1\%	$&$	66.5\%	$&$	68.6\%	$&$	69.9\%	$&$	69.5\%	$&$	68.8\%	$\\
$	\alpha = 0	$&$	48.1\%	$&$	59.5\%	$&$	66.4\%	$&$	68.8\%	$&$	70.0\%	$&$	69.9\%	$&$	69.1\%	$\\
$	\alpha = 0.2	$&$	49.0\%	$&$	61.2\%	$&$	67.0\%	$&$	68.6\%	$&$	70.1\%	$&$	69.7\%	$&$	68.9\%	$\\
$	\alpha = 0.8	$&$	51.9\%	$&$	65.9\%	$&$	65.3\%	$&$	67.9\%	$&$	69.1\%	$&$	68.8\%	$&$	68.1\%	$\\
$	\alpha = 1	$&$	53.2\%	$&$	65.7\%	$&$	65.3\%	$&$	67.2\%	$&$	68.1\%	$&$	67.8\%	$&$	67.3\%	$\\
$	\tau = 2	$&$	50.9\%	$&$	65.5\%	$&$	69.2\%	$&$	70.0\%	$&$	71.4\%	$&$	70.8\%	$&$	70.3\%	$\\
$	\tau = 5	$&$	50.9\%	$&$	65.2\%	$&$	67.5\%	$&$	69.0\%	$&$	70.4\%	$&$	70.0\%	$&$	69.4\%	$\\
$	\tau = 20	$&$	49.9\%	$&$	62.6\%	$&$	66.0\%	$&$	68.3\%	$&$	69.6\%	$&$	69.2\%	$&$	68.6\%	$\\
$	\tau = 30	$&$	49.5\%	$&$	61.9\%	$&$	65.9\%	$&$	68.3\%	$&$	69.7\%	$&$	69.4\%	$&$	68.7\%	$\\
$	\beta = 0	$&$	50.4\%	$&$	64.5\%	$&$	66.4\%	$&$	68.8\%	$&$	69.9\%	$&$	69.6\%	$&$	68.9\%	$\\
$	\beta = 0.2	$&$	50.3\%	$&$	64.4\%	$&$	66.5\%	$&$	68.7\%	$&$	70.0\%	$&$	69.6\%	$&$	69.1\%	$\\
$	\beta = 0.8	$&$	50.2\%	$&$	62.2\%	$&$	66.4\%	$&$	67.7\%	$&$	68.8\%	$&$	68.4\%	$&$	67.8\%	$\\
$	\beta = 1	$&$	50.3\%	$&$	59.1\%	$&$	65.6\%	$&$	66.4\%	$&$	67.5\%	$&$	67.3\%	$&$	66.5\%	$\\
$	\omega = 0.005	$&$	50.3\%	$&$	63.8\%	$&$	66.4\%	$&$	68.5\%	$&$	69.7\%	$&$	69.2\%	$&$	68.7\%	$\\
$	\omega = 0.02	$&$	50.7\%	$&$	63.9\%	$&$	66.3\%	$&$	68.4\%	$&$	69.7\%	$&$	69.5\%	$&$	68.9\%	$\\
$	\omega = 0.08	$&$	50.3\%	$&$	64.1\%	$&$	66.6\%	$&$	68.6\%	$&$	69.9\%	$&$	69.6\%	$&$	68.9\%	$\\
$	\omega = 0.15	$&$	50.3\%	$&$	64.3\%	$&$	66.6\%	$&$	68.8\%	$&$	69.9\%	$&$	69.6\%	$&$	69.0\%	$\\
\hline
\end{tabular}
\end{center}
\end{table}

\bibliographystyle{ormsv080}
\bibliography{myrefs}

\begin{thebibliography}{49}
\expandafter\ifx\csname natexlab\endcsname\relax\def\natexlab#1{#1}\fi
\expandafter\ifx\csname url\endcsname\relax
  \def\url#1{{\tt #1}}\fi
\expandafter\ifx\csname urlprefix\endcsname\relax\def\urlprefix{URL }\fi
\expandafter\ifx\csname urlstyle\endcsname\relax
  \expandafter\ifx\csname doi\endcsname\relax
  \def\doi#1{doi:\discretionary{}{}{}#1}\fi \else
  \expandafter\ifx\csname doi\endcsname\relax
  \def\doi{doi:\discretionary{}{}{}\begingroup \urlstyle{rm}\Url}\fi \fi

\bibitem[{Abdulkadiroglu and S{\"o}nmez(2013)}]{abdulkadiroglu2013matching}
Abdulkadiroglu, Atila, Tayfun S{\"o}nmez. 2013.
\newblock Matching markets: Theory and practice.
\newblock {\it Advances in Economics and Econometrics\/} {\bf 1} 3--47.

\bibitem[{Agrawal et~al.(2009)Agrawal, Wang, and Ye}]{agrawal2009dynamic}
Agrawal, Shipra, Zizhuo Wang, Yinyu Ye. 2009.
\newblock A dynamic near-optimal algorithm for online linear programming.
\newblock {\it arXiv preprint arXiv:0911.2974\/} .

\bibitem[{Akbarpour et~al.(2014)Akbarpour, Li, and
  Gharan}]{akbarpour2014dynamic}
Akbarpour, Mohammad, Shengwu Li, Shayan~Oveis Gharan. 2014.
\newblock Dynamic matching market design.
\newblock {\it arXiv preprint arXiv:1402.3643\/} .

\bibitem[{Alaei et~al.(2012)Alaei, Hajiaghayi, and Liaghat}]{alaei2012online}
Alaei, Saeed, MohammadTaghi Hajiaghayi, Vahid Liaghat. 2012.
\newblock Online prophet-inequality matching with applications to ad
  allocation.
\newblock {\it Proceedings of the 13th ACM Conference on Electronic
  Commerce\/}. ACM, 18--35.

\bibitem[{Anderson et~al.(2013)Anderson, Ashlagi, Kanoria, and
  Gamarnik}]{anderson2013efficient}
Anderson, Ross, Itai Ashlagi, Y~Kanoria, D~Gamarnik. 2013.
\newblock Efficient dynamic barter exchange.
\newblock Tech. rep., mimeo.

\bibitem[{Assadi et~al.(2015)Assadi, Hsu, and Jabbari}]{assadi2015online}
Assadi, Sepehr, Justin Hsu, Shahin Jabbari. 2015.
\newblock Online assignment of heterogeneous tasks in crowdsourcing markets.
\newblock {\it Third AAAI Conference on Human Computation and Crowdsourcing\/}.

\bibitem[{Babaioff et~al.(2008)Babaioff, Immorlica, Kempe, and
  Kleinberg}]{babaioff2008online}
Babaioff, Moshe, Nicole Immorlica, David Kempe, Robert Kleinberg. 2008.
\newblock Online auctions and generalized secretary problems.
\newblock {\it ACM SIGecom Exchanges\/} {\bf 7}(2) 7.

\bibitem[{Baccara et~al.(2015)Baccara, Lee, and Yariv}]{baccara2015optimal}
Baccara, Mariagiovanna, SangMok Lee, Leeat Yariv. 2015.
\newblock Optimal dynamic matching.
\newblock {\it Available at SSRN 2641670\/} .

\bibitem[{Bahmani and Kapralov(2010)}]{bahmani2010improved}
Bahmani, Bahman, Michael Kapralov. 2010.
\newblock Improved bounds for online stochastic matching.
\newblock {\it Algorithms--ESA 2010\/}. Springer, 170--181.

\bibitem[{Ball and Queyranne(2009)}]{ball2009toward}
Ball, Michael~O, Maurice Queyranne. 2009.
\newblock Toward robust revenue management: Competitive analysis of online
  booking.
\newblock {\it Operations Research\/} {\bf 57}(4) 950--963.

\bibitem[{Barr and Wohl(2013)}]{barr2013walmart}
Barr, A., J.~Wohl. 2013.
\newblock Wal-mart may get customers to deliver packages to online buyers.
\newblock \url{s http://
  www.reuters.com/article/2013/03/28/us-retail-walmart-delivery-idUSBRE92R03820130328}.

\bibitem[{Derman et~al.(1972)Derman, Lieberman, and
  Ross}]{derman1972sequential}
Derman, Cyrus, Gerald~J Lieberman, Sheldon~M Ross. 1972.
\newblock A sequential stochastic assignment problem.
\newblock {\it Management Science\/} {\bf 18}(7) 349--355.

\bibitem[{Devanur(2009)}]{Devanur09theadwords}
Devanur, Nikhil~R. 2009.
\newblock The adwords problem: Online keyword matching with budgeted bidders
  under random permutations.
\newblock {\it In Proc. 10th Annual ACM Conference on Electronic Commerge
  (EC\/}.

\bibitem[{Devanur et~al.(2011)Devanur, Jain, Sivan, and
  Wilkens}]{devanur2011near}
Devanur, Nikhil~R, Kamal Jain, Balasubramanian Sivan, Christopher~A Wilkens.
  2011.
\newblock Near optimal online algorithms and fast approximation algorithms for
  resource allocation problems.
\newblock {\it Proceedings of the 12th ACM conference on Electronic
  commerce\/}. ACM, 29--38.

\bibitem[{Dickerson et~al.(2018)Dickerson, Sankararaman, Srinivasan, and
  Xu}]{dickerson2018assigning}
Dickerson, John~P, Karthik~Abinav Sankararaman, Aravind Srinivasan, Pan Xu.
  2018.
\newblock Assigning tasks to workers based on historical data: Online task
  assignment with two-sided arrivals .

\bibitem[{Feldman et~al.(2010)Feldman, Henzinger, Korula, Mirrokni, and
  Stein}]{feldman2010online}
Feldman, Jon, Monika Henzinger, Nitish Korula, Vahab~S Mirrokni, Cliff Stein.
  2010.
\newblock Online stochastic packing applied to display ad allocation.
\newblock {\it Algorithms--ESA 2010\/}. Springer, 182--194.

\bibitem[{Feldman et~al.(2009)Feldman, Mehta, Mirrokni, and
  Muthukrishnan}]{feldman2009online}
Feldman, Jon, Aranyak Mehta, Vahab Mirrokni, S~Muthukrishnan. 2009.
\newblock Online stochastic matching: Beating 1-1/e.
\newblock {\it Foundations of Computer Science, 2009. FOCS'09. 50th Annual IEEE
  Symposium on\/}. IEEE, 117--126.

\bibitem[{Gallego et~al.(2015)Gallego, Li, Truong, and Wang}]{gallegoLTW2015}
Gallego, Guillermo, Anran Li, Van-Anh Truong, Xinshang Wang. 2015.
\newblock Online resource allocation with customer choice.
\newblock Working paper.

\bibitem[{Goel and Mehta(2008)}]{goel2008online}
Goel, Gagan, Aranyak Mehta. 2008.
\newblock Online budgeted matching in random input models with applications to
  adwords.
\newblock {\it Proceedings of the nineteenth annual ACM-SIAM symposium on
  Discrete algorithms\/}. Society for Industrial and Applied Mathematics,
  982--991.

\bibitem[{Haeupler et~al.(2011)Haeupler, Mirrokni, and
  Zadimoghaddam}]{haeupler2011online}
Haeupler, Bernhard, Vahab~S Mirrokni, Morteza Zadimoghaddam. 2011.
\newblock Online stochastic weighted matching: Improved approximation
  algorithms.
\newblock {\it Internet and Network Economics\/}. Springer, 170--181.

\bibitem[{Hassan and Curry(2014)}]{hassan2014multi}
Hassan, Umair~Ul, Edward Curry. 2014.
\newblock A multi-armed bandit approach to online spatial task assignment.
\newblock {\it Ubiquitous Intelligence and Computing, 2014 IEEE 11th Intl Conf
  on and IEEE 11th Intl Conf on and Autonomic and Trusted Computing, and IEEE
  14th Intl Conf on Scalable Computing and Communications and Its Associated
  Workshops (UTC-ATC-ScalCom)\/}. IEEE, 212--219.

\bibitem[{Hill and Kertz(1983)}]{hill1983stop}
Hill, Theodore~P, Robert~P Kertz. 1983.
\newblock Stop rule inequalities for uniformly bounded sequences of random
  variables.
\newblock {\it Transactions of the American Mathematical Society\/} {\bf
  278}(1) 197--207.

\bibitem[{Ho and Vaughan(2012)}]{ho2012online}
Ho, Chien-Ju, Jennifer~Wortman Vaughan. 2012.
\newblock Online task assignment in crowdsourcing markets.
\newblock {\it AAAI\/}, vol.~12. 45--51.

\bibitem[{Hu and Zhou(2015)}]{hu2015dynamic}
Hu, Ming, Yun Zhou. 2015.
\newblock Dynamic matching in a two-sided market.
\newblock {\it Available at SSRN\/} .

\bibitem[{Jaillet and Lu(2013)}]{jaillet2013online}
Jaillet, Patrick, Xin Lu. 2013.
\newblock Online stochastic matching: New algorithms with better bounds.
\newblock {\it Mathematics of Operations Research\/} {\bf 39}(3) 624--646.

\bibitem[{Karande et~al.(2011)Karande, Mehta, and Tripathi}]{karande2011online}
Karande, Chinmay, Aranyak Mehta, Pushkar Tripathi. 2011.
\newblock Online bipartite matching with unknown distributions.
\newblock {\it Proceedings of the forty-third annual ACM symposium on Theory of
  computing\/}. ACM, 587--596.

\bibitem[{Karp et~al.(1990)Karp, Vazirani, and Vazirani}]{karp1990optimal}
Karp, R.~M., U.~V. Vazirani, V.~V. Vazirani. 1990.
\newblock An optimal algorithm for on-line bipartite matching.
\newblock {\it Proceedings of the Twenty-second Annual ACM Symposium on Theory
  of Computing\/}. STOC '90, ACM, New York, NY, USA, 352--358.
\newblock \doi{10.1145/100216.100262}.
\newblock \urlprefix\url{http://doi.acm.org/10.1145/100216.100262}.

\bibitem[{Kleinberg(2005)}]{kleinberg2005multiple}
Kleinberg, Robert. 2005.
\newblock A multiple-choice secretary algorithm with applications to online
  auctions.
\newblock {\it Proceedings of the sixteenth annual ACM-SIAM symposium on
  Discrete algorithms\/}. Society for Industrial and Applied Mathematics,
  630--631.

\bibitem[{Krengel and Sucheston(1977)}]{krengel1977}
Krengel, Ulrich, Louis Sucheston. 1977.
\newblock Semiamarts and finite values.
\newblock {\it Bull. Amer. Math. Soc.\/} {\bf 83}(4) 745--747.
\newblock \urlprefix\url{https://projecteuclid.org:443/euclid.bams/1183538915}.

\bibitem[{Krengel and Sucheston(1978)}]{krengel1978semiamarts}
Krengel, Ulrich, Louis Sucheston. 1978.
\newblock On semiamarts, amarts, and processes with finite value.
\newblock {\it Advances in Prob\/} {\bf 4} 197--266.

\bibitem[{Mahdian and Yan(2011)}]{mahdian2011online}
Mahdian, Mohammad, Qiqi Yan. 2011.
\newblock Online bipartite matching with random arrivals: an approach based on
  strongly factor-revealing lps.
\newblock {\it Proceedings of the forty-third annual ACM symposium on Theory of
  computing\/}. ACM, 597--606.

\bibitem[{Manshadi et~al.(2012)Manshadi, Gharan, and
  Saberi}]{manshadi2012online}
Manshadi, Vahideh~H, Shayan~Oveis Gharan, Amin Saberi. 2012.
\newblock Online stochastic matching: Online actions based on offline
  statistics.
\newblock {\it Mathematics of Operations Research\/} {\bf 37}(4) 559--573.

\bibitem[{Mehta(2012)}]{mehta2012online}
Mehta, Aranyak. 2012.
\newblock Online matching and ad allocation.
\newblock {\it Theoretical Computer Science\/} {\bf 8}(4) 265--368.

\bibitem[{Molinaro and Ravi(2013)}]{molinaro2013geometry}
Molinaro, Marco, R~Ravi. 2013.
\newblock The geometry of online packing linear programs.
\newblock {\it Mathematics of Operations Research\/} {\bf 39}(1) 46--59.

\bibitem[{Pofeldt(2016)}]{pofeldt2016freelancers}
Pofeldt, Elaine. 2016.
\newblock Freelancers now make up 35\% of u.s. workforce.
\newblock {\it Forbes\/} .

\bibitem[{Qin et~al.(2015)Qin, Zhang, Hua, and Shi}]{CongApproximationRM}
Qin, Chao, Huanan Zhang, Cheng Hua, Cong Shi. 2015.
\newblock A simple admission control policy for revenue management problems
  with non-stationary customer arrivals.
\newblock {\it working paper\/} .

\bibitem[{Rinott et~al.(1987)Rinott, Samuel-Cahn
  et~al.}]{rinott1987comparisons}
Rinott, Yosef, Ester Samuel-Cahn, et~al. 1987.
\newblock Comparisons of optimal stopping values and prophet inequalities for
  negatively dependent random variables.
\newblock {\it The Annals of Statistics\/} {\bf 15}(4) 1482--1490.

\bibitem[{Samuel-Cahn(1991)}]{samuel1991prophet}
Samuel-Cahn, Ester. 1991.
\newblock Prophet inequalities for bounded negatively dependent random
  variables.
\newblock {\it Statistics \& probability letters\/} {\bf 12}(3) 213--216.

\bibitem[{Singer and Mittal(2013)}]{singer2013pricing}
Singer, Yaron, Manas Mittal. 2013.
\newblock Pricing mechanisms for crowdsourcing markets.
\newblock {\it Proceedings of the 22nd international conference on World Wide
  Web\/}. ACM, 1157--1166.

\bibitem[{Singla and Krause(2013)}]{singla2013truthful}
Singla, Adish, Andreas Krause. 2013.
\newblock Truthful incentives in crowdsourcing tasks using regret minimization
  mechanisms.
\newblock {\it Proceedings of the 22nd international conference on World Wide
  Web\/}. ACM, 1167--1178.

\bibitem[{Spivey and Powell(2004)}]{spivey2004dynamic}
Spivey, Michael~Z, Warren~B Powell. 2004.
\newblock The dynamic assignment problem.
\newblock {\it Transportation Science\/} {\bf 38}(4) 399--419.

\bibitem[{Stein et~al.(2016)Stein, Truong, and Wang}]{SteinTW2015}
Stein, Clifford, Van-Anh Truong, Xinshang Wang. 2016.
\newblock Advance service reservation with heterogeneous customers.
\newblock Working paper.

\bibitem[{Su and Zenios(2005)}]{su2005patient}
Su, Xuanming, Stefanos~A Zenios. 2005.
\newblock Patient choice in kidney allocation: A sequential stochastic
  assignment model.
\newblock {\it Operations research\/} {\bf 53}(3) 443--455.

\bibitem[{Subramanian et~al.(2015)Subramanian, Kanth, Moharir, and
  Vaze}]{subramanian2015online}
Subramanian, Ashwin, G~Sai Kanth, Sharayu Moharir, Rahul Vaze. 2015.
\newblock Online incentive mechanism design for smartphone crowd-sourcing.
\newblock {\it Modeling and Optimization in Mobile, Ad Hoc, and Wireless
  Networks (WiOpt), 2015 13th International Symposium on\/}. IEEE, 403--410.

\bibitem[{Talluri and Van~Ryzin(2004)}]{TalluriV2004}
Talluri, K., G.~Van~Ryzin. 2004.
\newblock {Revenue management under a general discrete choice model of consumer
  behavior}.
\newblock {\it Management Science\/} {\bf 50}(1) 15--33.

\bibitem[{Tong et~al.(2016)Tong, She, Ding, Wang, and Chen}]{tong2016online}
Tong, Yongxin, Jieying She, Bolin Ding, Libin Wang, Lei Chen. 2016.
\newblock Online mobile micro-task allocation in spatial crowdsourcing.
\newblock {\it Data Engineering (ICDE), 2016 IEEE 32nd International Conference
  on\/}. IEEE, 49--60.

\bibitem[{van Ryzin and Talluri(2005)}]{van2005introduction}
van Ryzin, Garrett~J, Kalyan~T Talluri. 2005.
\newblock An introduction to revenue management.
\newblock {\it Tutorials in operations research\/}  142--195.

\bibitem[{Wang et~al.(2015)Wang, Truong, and Bank}]{wangTB2015}
Wang, Xinshang, Van-Anh Truong, David Bank. 2015.
\newblock Online advance admission scheduling for services, with customer
  preferences.
\newblock Working paper.

\bibitem[{Zhao et~al.(2014)Zhao, Li, and Ma}]{zhao2014crowdsource}
Zhao, Dong, Xiang-Yang Li, Huadong Ma. 2014.
\newblock How to crowdsource tasks truthfully without sacrificing utility:
  Online incentive mechanisms with budget constraint.
\newblock {\it IEEE INFOCOM 2014-IEEE Conference on Computer Communications\/}.
  IEEE, 1213--1221.

\end{thebibliography}

\end{document}